\documentclass[a4paper, 10pt]{article}

\usepackage{algorithmic}
\usepackage{algorithm}
\usepackage{amsmath}
\usepackage{amsfonts}
\usepackage{amssymb}
\usepackage{amsthm}

\usepackage{graphicx}
\usepackage[caption=false]{subfig}
\usepackage[utf8]{inputenc}

\usepackage{siunitx} 
\usepackage{booktabs}

\usepackage{setspace}
\usepackage[displaymath, mathlines]{lineno}

\newtheorem{theorem}{Theorem}

\newtheorem{lemma}{Lemma}
\newtheorem{property}{Property}

\newtheorem{corollary}{Corollary}

\newtheorem*{lemma_el_new}{Lemma~\ref{lemma_el_new}}
\newtheorem*{lemma_mono}{Lemma~\ref{lemma_mono}}
\newtheorem*{lemma_mu}{Lemma~\ref{lemma_mu}}
\newtheorem*{lemma_posdef}{Lemma~\ref{lemma_posdef}}
\newtheorem*{lemma_xk}{Lemma~\ref{lemma_xk}}

\newtheorem{innercustomgeneric}{\customgenericname}
\providecommand{\customgenericname}{}
\newcommand{\newcustomtheorem}[2]{%
  \newenvironment{#1}[1]
  {%
   \renewcommand\customgenericname{#2}%
   \renewcommand\theinnercustomgeneric{##1}%
   \innercustomgeneric
  }
  {\endinnercustomgeneric}
}

\newcustomtheorem{customthm}{Theorem}
\newcustomtheorem{customlemma}{Lemma}

\usepackage{comment}
\usepackage[hidelinks]{hyperref} 
\usepackage[a4paper,margin=2.54cm]{geometry}
\title{Quadratic nonseparable resource allocation problems with generalized bound constraints}
\author{Martijn H. H. Schoot Uiterkamp, Marco E. T. Gerards, Johann L. Hurink
\\
University of Twente}

\usepackage{multicol}

\usepackage{hyperref}
\makeatletter
\newcommand*{\textlabel}[2]{%
  \edef\@currentlabel{#1}
  \phantomsection
  #1\label{#2}
}
\makeatother

\begin{document}
\maketitle

\begin{abstract}
We study a quadratic nonseparable resource allocation problem that arises in the area of decentralized energy management (DEM), where unbalance in electricity networks has to be minimized. In this problem, the given resource is allocated over a set of activities that is divided into subsets, and a cost is assigned to the overall allocated amount of resources to activities within the same subset. We derive two efficient algorithms with $O(n\log n)$ worst-case time complexity to solve this problem. For the special case where all subsets have the same size, one of these algorithms even runs in linear time given the subset size. Both algorithms are inspired by well-studied breakpoint search methods for separable convex resource allocation problems. Numerical evaluations on both real and synthetic data confirm the theoretical efficiency of both algorithms and demonstrate their suitability for integration in DEM systems.
\end{abstract}

\section{Introduction}
\label{sec_intro}
Resource allocation problems belong to the fundamental problems in the operations research literature. These problems involve the allocation of a given resource (e.g., money or energy) over a set of activities (e.g., projects or time slots) while minimizing a given cost function or maximizing a given utility function. In its simplest form, the problem can be formulated mathematically as follows:
\begin{align*}
\text{RAP} : \
\min_{x \in \mathbb{R}^n} \
&
\sum_{i=1}^n f_i(x_i) \\
\text{s.t. }
& \sum_{i=1}^n x_i = R \\
& l_i \leq x_i \leq u_i, \quad i \in \lbrace 1,\ldots,n \rbrace
\end{align*}
Here, each variable~$x_i$ represents the amount of the total resource~$R \in \mathbb{R}$ that is allocated to activity~$i$ and the values $l_i,u_i \in \mathbb{R}$ are lower and upper bounds on the amount allocated to activity~$i$. Moreover, each function $f_i : \mathbb{R} \rightarrow \mathbb{R}$ assigns a cost to allocating resource to activity~$i$.

In this article, we study the following more specific allocation problem, which is an extension of the quadratic resource allocation problem:
\begin{subequations}
       \makeatletter
        \def\@currentlabel{QRAP-NonSep-GBC}
        \makeatother
        \renewcommand{\theequation}{1\alph{equation}}
\label{prob} 
\begin{align}
\text{QRAP-NonSep-GBC} : \
\min_{x \in \mathbb{R}^n} \ & \sum_{j=1}^m \frac{1}{2} w_j \left( \sum_{i \in \mathcal{N}_j}  x_i \right)^2 + \sum_{i=1}^n \left(\frac{1}{2}a_i x_i^2 + b_i x_i \right) \label{eq_p_01} \\
\text{s.t. } & \sum_{i=1}^n  x_i = R \label{eq_p_02} \\
& L_j \leq \sum_{i \in \mathcal{N}_j} x_i \leq U_j, \quad j \in \lbrace 1,\ldots,m \rbrace \label{eq_p_03} \\
& l_i \leq x_i \leq u_i, \quad i \in \lbrace 1,\ldots,n \rbrace. \label{eq_p_04}
\end{align} 
\end{subequations}
where $w,b,l,u \in \mathbb{R}^n$, $a \in \mathbb{R}^n_{> 0}$, $R \in \mathbb{R}$, and $L,U \in \mathbb{R}^m$ are given inputs. Furthermore, in this problem, a partition of the index set $\mathcal{N} := \lbrace 1,\ldots,n\rbrace$ into $m$ disjoint subsets $\mathcal{N}_1,\ldots,\mathcal{N}_m$ of size $n_1,\ldots,n_j$ indexed by $\mathcal{M} := \lbrace 1,\ldots,m \rbrace$ is given. The objective function of Problem~\ref{prob} assigns for each subset $\mathcal{N}_j$ a cost to the sum of all allocated amounts associated with this subset and to the individual amounts. Similarly, Constraints~(\ref{eq_p_03}) and~(\ref{eq_p_04}) put bounds on the sum of all variables associated with each given subset and and on the individual variables.

Our interest in studying this problem stems from its application in decentralized energy management (DEM). The aim of DEM is to optimize the simultaneous energy consumption of multiple devices within a neighborhood. Compared to other energy management paradigms such as \emph{centralized} energy management, within a DEM system devices optimize their own consumption locally and the control system coordinates the local optimization of these devices to optimize certain neighborhood objectives.

In particular, we are interested in the local optimization of a specific device class within DEM, namely the scheduling of electric vehicles (EVs) that are equipped with a three-phase charger. This means that the EV can distribute its charging arbitrarily over all the three phases of the low-voltage network. Recent studies show that three-phase EV charging, as opposed to single-phase EV charging, can reduce losses in the electricity grid, reduce the stress on grid assets, and thereby prevent outages caused by a high penetration of EVs charging simultaneously on a single phase (\cite{Weckx2015, SchootUiterkamp2017}). We discuss this issue in more detail in Section~\ref{sec_mot} and we show that the three-phase EV charging problem can be modeled as an instance of Problem~\ref{prob}.

An important aspect of the DEM paradigm is that device-level problems, such as the aforementioned three-phase EV charging problem, are solved locally. This means that the corresponding device-level optimization algorithms are executed on embedded systems located within, e.g., households or the charging equipment. It is important that these algorithms are very efficient with regard to both execution time and memory, since often they are called multiple times within the DEM system and the embedded systems on which the algorithms run have limited computational power and memory (see, e.g., \cite{Beaudin2015}). Therefore, efficient and tailored device-level optimization algorithms are crucial ingredients for the real-life implementation of DEM systems. In particular, to solve the three-phase EV charging problem, an efficient algorithm to solve Problem~\ref{prob} is required. 

For more background on DEM we refer to \cite{Siano2014, Esther2016}. Other applications of Problem~\ref{prob} are in the areas of, e.g., portfolio optimization (see, e.g., \cite{Lobo2007}), stratified sampling \cite{Sanathanan1971}, and transportation problems (see, e.g., \cite{Cosares1994}).

Problem~\ref{prob} can be classified as a quadratic \emph{nonseparable} resource allocation problem with generalized bound constraints (Constraint~(\ref{eq_p_03})). The nonseparability is due to the terms $(\sum_{i \in \mathcal{N}_j} x_i)^2$, which cannot be written as the sum of single-variable functions and are thus nonseparable. When the factors $w_j$ are zero, these nonseparable terms disappear and Problem~\ref{prob} becomes the quadratic \emph{separable} resource allocation problem with generalized bound constraints. In the literature, this problem has hardly been studied: a special case that includes only generalized upper bound constraints is studied in \cite{Hochbaum1995} and \cite{Bretthauer1997}. When in addition the generalized bound constraints are omitted, Problem~\ref{prob} reduces to the quadratic \emph{simple} separable resource allocation problem. This problem and its extension to convex cost functions has been well-studied (see, e.g., \cite{Patriksson2008, Patriksson2015} and the references therein).

Observe that Problem~\ref{prob} can be modeled as a minimum convex quadratic cost flow problem if $w \geq 0$. Therefore, this case can be solved in strongly polynomial time \cite{Vegh2016}. In fact, since its network structure is series-parallel, it can be solved by the algorithms in \cite{Tamir1993} and \cite{Moriguchi2011} in $O(n^2)$ time. However, when some of the factors $w_j$ are negative, existing approaches for solving this type of flow problem do not apply anymore. In particular, this holds for the aforementioned EV scheduling problem in DEM, where the objective of minimizing load unbalance is modeled as an instance of Problem~\ref{prob} by setting some or all of the factors $w_j$ to a negative number (see also Section~\ref{sec_prob_EV}).

In this article, we present two $O(n \log n)$ time algorithms for strictly convex instances of Problem~\ref{prob}, thereby adding a new problem to the small class of quadratic programming problems that can be solved efficiently in strongly polynomial time. For this, we derive a property of problem instances that uniquely characterizes the class of strictly convex instances to the problem. This class includes problems in which some or all of the factors $w_j$ are negative and, in particular, includes the three-phase EV charging problem. Our algorithms are, in their essence, breakpoint search algorithms. This type of algorithm is commonly used to solve separable resource allocation problems. Such algorithms consider the Lagrangian dual of the original problem and exploit the structure of the Karush-Kuhn-Tucker (KKT) optimality conditions to efficiently search for the optimal (dual) multiplier associated with the resource constraint~(\ref{eq_p_02}). We show that for (strictly) convex instances of Problem~\ref{prob}, these conditions can be exploited in a similar way. 

For the case where all subsets $\mathcal{N}_j$ have the same size, i.e., where all $n_j$'s are equal to some constant $C$, we show that one of the derived algorithms runs in $O(n \log C)$ time, i.e., given~$C$ this algorithm has a linear time complexity. Thereby, we add a new problem to the (even smaller) class of quadratic programming problems that can be solved in linear time and we show that the three-phase EV charging problem can be solved in $O(n)$ time. Furthermore, we show for the special case where all weights $w_j$ are zero, i.e., the quadratic separable resource allocation problem with generalized bound constraints, that both Problem~\ref{prob} and its version with integer variables can be solved in $O(n)$ time. Although the version with integer variables is not the main focus of this article, it may be of independent interest for research on general resource allocation problems where often both the continuous and integer version of a given resource allocation problem are studied in parallel (see, e.g., \cite{Hochbaum1994,Moriguchi2011}). 

We evaluate the performance of our algorithms on both realistic instances of the three-phase EV charging problem and synthetically generated instances of different sizes. These evaluations suggest that our algorithms are suitable for integration in DEM systems since they are fast and do not require much memory. Furthermore, they show that our algorithms scale well when the number~$m$ of subsets or the subset sizes~$n_j$ increases, i.e., the evolution of their execution time matches the theoretical worst-case complexity of $O(n \log n)$. In fact, we show that our algorithms are capable of outperforming the commercial solver MOSEK by two orders of magnitude for instances of up to 1 million variables.

The remainder of this article is organized as follows. In Section~\ref{sec_mot}, we explain in more detail the application of Problem~\ref{prob} in DEM and, specifically, in three-phase EV scheduling. In Section~\ref{sec_analysis}, we analyze the structure of Problem~\ref{prob} and derive a crucial property of feasible solutions to the problem. We use this property to derive our solution approach to solve Problem~\ref{prob} in Section~\ref{sec_sol} and in Section~\ref{sec_alg_two}, we present two $O(n \log n)$ algorithms based on this approach. In Section~\ref{sec_eval}, we evaluate the performance of our algorithms and, finally, Section~\ref{sec_concl} contains some concluding remarks.

Summarizing, the contributions of this article are as follows:
\begin{enumerate}
\item
We derive two $O(n \log n)$ time algorithms for Problem~\ref{prob}. In contrast to existing work \cite{Tamir1993,Moriguchi2011}, this algorithm can be applied to all strictly convex instances of Problem~\ref{prob}, even those where some or all of the factors $w_j$ are negative.
\item
For the special case where all subsets $\mathcal{N}_j$ have the same size~$C$, we show that one of our algorithms runs in linear time given~$C$, hereby extending the small class of quadratic programming problems that are solvable in linear time.
\item Our algorithm solves an important problem in DEM and can make a significant impact on the integration of EVs in residential distribution grids.
\end{enumerate}

\section{Motivation}
\label{sec_mot}
In this section, we describe in more detail our motivation for studying Problem~\ref{prob}. For this, Section~\ref{sec_3} provides a short introduction to load balancing in three-phase electricity networks and discusses the relevance of minimizing load unbalance. In Section~\ref{sec_prob_EV}, we formulate the three-phase EV charging problem and show that this problem is an instance of Problem~\ref{prob}.

\subsection{Load balancing in three-phase electricity networks}
\label{sec_3}

Load balancing has as goal to distribute the power consumption of a neighborhood over a given time horizon and over the three phases of the low-voltage network such that peak consumption and unbalance between phases is minimized. Peak consumption occurs when the consumption is not spread out equally over the time horizon but instead is concentrated within certain time periods. This is generally seen as non-desirable since it induces an increase in energy losses, stress on grid assets such as transformers, and can even lead to outages (see, e.g., \cite{Hoogsteen2017b}). As a consequence, many DEM systems in the literature take into account the minimization of peak consumption when scheduling, e.g., EV charging (see, e.g., \cite{Gan2013,Gerards2015,Mou2015}).

However, minimization of load unbalance between phases is hardly considered in optimization approaches for EV scheduling. To explain the relevance of load unbalance minimization, in the following we first consider three-phase electricity networks in general (for a more detailed and comprehensive introduction to this topic, we refer to \cite{Stevenson1975} and \cite{AAC}). 

In residential electricity distribution networks (or, more generally, low-voltage networks), electrical energy is transported by electrical current that flows through a conductor (e.g., a wire). This current can be seen as a signal with a given frequency and amplitude, which leads to (alternating current) power, i.e.,  the average energy transported in each cycle. In principle, only one supply conductor is required to transport electrical energy between two points. However, it is more efficient to divide this energy over three bundled conductors whose currents have the same frequency but an equidistance phase shift. This means that there is a phase difference of 120 degrees between each pair of conductors. Networks wherein the conductors are bundled in this way are referred to as \emph{three-phase} networks, where the term ``phase'' generally refers to one of the three bundled conductors.  Figure~\ref{fig_phase_01} illustrates the concept of three-phase systems.

In order to maximize the efficiency of a three-phase network, ideally the power consumption from all three phases is equal. When this is not the case, negative effects similar to those of peak consumption can occur, i.e., energy losses, wearing of grid assets, and outages. With the increasing penetration of EVs in the low-voltage network, actively maintaining load balance becomes important. This is mainly because the power consumption of an EV is in general much larger than the average power consumption of a household (see, e.g., \cite{Schootuiterkamp2016}) and most EVs, especially in the Netherlands, are connected to only one of the three phases. As a consequence, when charging multiple EVs simultaneously, large load unbalance can occur when the (charging of the) EVs are (is) not divided equally over the phases \cite{Hoogsteen2017b}.

Recently, \cite{Weckx2015} and \cite{SchootUiterkamp2017} explored the potential of three-phase EV charging, i.e., allowing an EV to distribute its charging over the three phases for minimizing load unbalance. Both works suggest that three-phase EV charging can significantly reduce the distribution losses and stress on the grid compared to single-phase EV charging, even when using the same DEM methodology.

\begin{figure}[ht!]
\centering
\caption{Schematic view of the three-phase system. $I_1$, $I_2$ and $I_3$ represent the current on each of the three phases and $\phi_1$, $\phi_2$, and $\phi_3$ represent the phase angles (with regard to the horizontal axis). The light gray arrows represent a balanced load distribution, whereas the black arrows represent load unbalance.} \label{fig_phase_01}
{\includegraphics{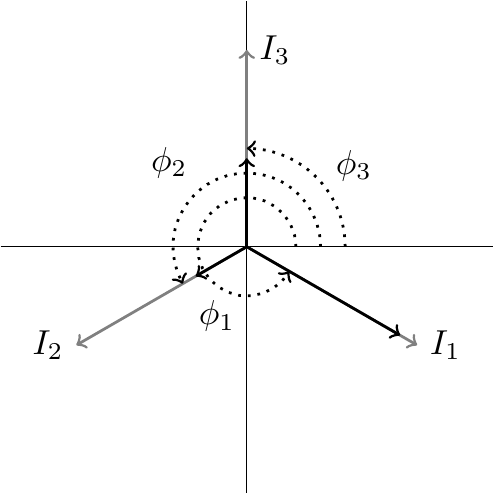}}
\end{figure}

\subsection{Modeling the three-phase EV charging problem}
\label{sec_prob_EV}

The problem of three-phase EV charging with the objective to minimize peak consumption and load unbalance can be modeled as an instance of Problem~\ref{prob}. For this, we consider a division of the scheduling horizon into~$m$ equidistant time intervals of length $\Delta t$ labeled according to $\mathcal{M}:= \lbrace 1,\ldots,m \rbrace$. Furthermore, we define the set $\mathcal{P} := \lbrace 1,2,3 \rbrace$ as the set of phases. We introduce for each $j \in \mathcal{M}$ and $p \in \mathcal{P}$ the variable $z_{j,p}$ that denotes the power consumption of the EV drawn from phase~$p$ during time interval~$j$. Moreover, we denote by $q_{j,p}$ be the remaining household power consumption drawn from phase~$p$ during interval~$j$. This consumption is assumed to be known. Furthermore, we assume that we know on forehand the total required energy that must be charged by the EV and denote this requirement by~$\tilde{R}$. Finally, we denote the minimum and maximum allowed power consumption from phase~$p$ during interval~$j$ by $\tilde{l}_{j,p}$ and $\tilde{u}_{j,p}$ respectively and the minimum and maximum allowed consumption from all three phases summed together by $\tilde{L}_j$ and $\tilde{U}_j$ respectively.

The objective of minimizing peak consumption can be achieved by ``flattening out'' the overall consumption as much as possible over the time intervals. Thus, noting that the term $\sum_{p =1}^3 (q_{j,p} + z_{j,p})$ represents the total power consumption during interval~$j$, we model this objective by minimizing the function
\begin{equation*}
\sum_{j=1}^m \left( \sum_{p =1}^3 (q_{j,p} + z_{j,p}) \right)^2 .
\end{equation*}
For minimizing load unbalance, we aim to equally distribute the consumption during each time interval~$j$ over the three phases. We can model the objective of minimizing load unbalance by minimizing the function
\begin{equation}
\sum_{j=1}^m \left( \frac{3}{2} \sum_{p=1}^3 (q_{j,p} + z_{j,p})^2 -\frac{1}{2} \left(\sum_{p=1}^3 (q_{j,p} + z_{j,p}) \right)^2 \right)
\label{eq_obj_phase}
\end{equation}
(see Appendix~\ref{app_obj} for the derivation of this expression). This leads to the following optimization problem that we denote by \textlabel{EV-3Phase}{prob_EV}:
\begin{align*}
\text{EV-3Phase}: \
\min_{z \in \mathbb{R}^{m \times 3}} \ & W_1 \sum_{j=1}^m \left( \sum_{p=1}^3 (q_{j,p} + z_{j,p}) \right)^2 + W_2   \sum_{j=1}^m \left( \frac{3}{2} \sum_{p=1}^3 (q_{j,p} + z_{j,p})^2 - \frac{1}{2} \left(\sum_{p=1}^3 (q_{j,p} + z_{j,p}) \right)^2 \right) \\
\text{s.t. } & \sum_{j=1}^m \sum_{p=1}^3 z_{j,p} \Delta t =  \tilde{R}, \\
& \tilde{L}_j \leq \sum_{p=1}^3 z_{j,p}  \leq \tilde{U}_j, \quad j \in \mathcal{M} \\
& \tilde{l}_{j,p} \leq z_{j,p} \leq \tilde{u}_{j,p}, \quad j \in \mathcal{M}, \ p \in \mathcal{P}.
\end{align*}
Here, $W_1$ and $W_2$ are positive weights that express the trade-off between the two objectives. By choosing the parameters as given in Table~\ref{tab_01}, this problem becomes an instance of Problem~\ref{prob} (see also Appendix~\ref{app_obj}). Observe that if $W_2 > 2W_1$, the weights $w_j$ are negative and thus Problem~\ref{prob_EV} cannot be solved as a minimum convex quadratic cost flow problem using, e.g., the algorithms in \cite{Tamir1993} and \cite{Moriguchi2011}.
\begin{table}[ht!]
\centering
\begin{tabular}{ll | ll}
\toprule
\multicolumn{2}{l |}{Parameter / variable in Problem~\ref{prob}} & \multicolumn{2}{l}{Parameter / variable in Problem~\ref{prob_EV}} \\
\midrule
$\mathcal{N}_j$, &$j \in \mathcal{M}$ & $\mathcal{P} := \lbrace 1,2,3 \rbrace$ \\
$(x_i)_{i \in \mathcal{N}_j}$, &$j \in \mathcal{M}$ & $(z_{j,p})_{p \in \mathcal{P}}$, & $j \in \mathcal{M}$ \\
$w_j$, &$j \in \mathcal{M}$ & $2 W_1 - W_2$ \\
$a_i$, & $i \in \mathcal{N}$ & $3W_2$ \\
$b_i$, & $j \in \mathcal{M}$, $i \in \mathcal{N}_j$ & $\left(W_1 - \frac{1}{2} W_2 \right) \sum_{p=1}^3  q_{j,p} + \frac{3}{2}W_2 q_{j,p}$ \\
$R$ && $\frac{\tilde{R}}{\Delta t}$ \\
$(l_i)_{i \in \mathcal{N}_j}$, &$j \in \mathcal{M}$ & $(\tilde{l}_{j,p})_{p \in \mathcal{P}}$ \\
$(u_i)_{i \in \mathcal{N}_j}$, &$j \in \mathcal{M}$ & $(\tilde{u}_{j,p})_{p \in \mathcal{P}}$ \\
$L_j$, & $j \in \mathcal{M}$ & $\tilde{L}_j$, &$j \in \mathcal{M}$ \\
$U_j$, & $j \in \mathcal{M}$ & $\tilde{U}_j$, &$j \in \mathcal{M}$ \\
\bottomrule
\end{tabular}
\caption{Modeling Problem~\ref{prob_EV} as an instance of Problem~\ref{prob}.}
\label{tab_01}
\end{table}

\section{Analysis}
\label{sec_analysis}

In this section, we consider the general version of Problem~\ref{prob} and derive some of its properties. First, in Section~\ref{sec_convex}, we derive a necessary and sufficient condition on the vectors $w$ and $a$ for strict convexity of Problem~\ref{prob}. Moreover, we show that the three-phase EV charging problem as presented in Section~\ref{sec_prob_EV} satisfies this condition. Second, in Section~\ref{sec_replace}, we show that we may replace Constraint~(\ref{eq_p_03}) by equivalent single-variable constraints without changing the optimal solution to the problem. This greatly simplifies the derivation of our solution approach in Section~\ref{sec_sol}. Third, in Section~\ref{sec_KKT}, we derive a property of the structure of optimal solutions to Problem~\ref{prob} that forms the crucial ingredient for our solution approach to solve the problem.

\subsection{Convex instances of Problem~\ref{prob}}
\label{sec_convex}
Since all constraints of Problem~\ref{prob} are linear, the problem is strictly convex if and only if the second-derivative matrix (the Hessian) of its objective function is positive definite. Since this objective function is separable over the indices $j$, it suffices to investigate for each $j \in \mathcal{M}$ separately if the function
\begin{equation*}
f_j ((x_i)_{i \in \mathcal{N}_j}) := \frac{1}{2} w_j \left( \sum_{i \in \mathcal{N}_j} x_i \right)^2 + \sum_{i \in \mathcal{N}_j} \left(\frac{1}{2} a_i x_i^2 + b_i x_i \right)
\end{equation*}
is strictly convex. We do this by checking whether the Hessian $H^j$ of $f_j$ is positive definite. This Hessian is given by
\begin{equation*}
H^j := w_j e e^{\top} + \text{diag}(a^j),
\end{equation*}
where $e$ is the vector of ones of appropriate size and $a^j := (a_i)_{i \in \mathcal{N}_j}$. Lemma~\ref{lemma_posdef} provides a characterization for which choices of $w_j$ and $a^j$ the Hessian $H^j$ is positive definite. This characterization can also be obtained as a special case of Theorem~1 in~\cite{Spedicato1975}.

\begin{lemma}
$H^j$ is positive definite if and only if $1 + w_j \sum_{i \in \mathcal{N}_j} 1/a_{i} > 0$.
\label{lemma_posdef}
\end{lemma}
\begin{proof}
See Appendix~\ref{sec_app_lemma_posdef}.
\end{proof}

Lemma~\ref{lemma_posdef} implies that an instance of Problem~\ref{prob} is strictly convex if and only if $1+ w_j \sum_{i' \in \mathcal{N}_j} 1/a_{i'} > 0$ for each $j \in \mathcal{M}$. To stress the importance of this relation and for future reference, we state this relation as a property:
\begin{property}
For each $j \in \mathcal{M}$, it holds that $1+ w_j \sum_{i' \in \mathcal{N}_j} 1/a_{i'} > 0$.
\label{prop_01}
\end{property}
\noindent
For the remainder of this article, we consider only instances of Problem~\ref{prob} that satisfy Property~\ref{prop_01}. We conclude this subsection by observing that the parameters for Problem~\ref{prob_EV} satisfy this property:
\begin{equation*}
1+ w_j \sum_{i \in \mathcal{N}_j} \frac{1}{a_{i}} 
= 1+ (2W_1 - W_2) \sum_{p=1}^3 \frac{1}{3W_2}
= 1 + \frac{2W_1 - W_2}{W_2}
= \frac{2W_1}{W_2} > 0.
\end{equation*}

\subsection{Constraint elimination}
\label{sec_replace}
In Section~\ref{sec_convex}, we studied properties of the objective function of Problem~\ref{prob}. In contrast, we focus in this section on properties of the \emph{constraints} of Problem~\ref{prob}. For this, note that it is the addition of the lower and upper bound constraints~(\ref{eq_p_03}) that make the constraint set of Problem~\ref{prob} complex compared to the constraint set of the original resource allocation problem~RAP. Therefore, the goal of this section is to reduce this complexity. More precisely, in this section, we show that we can replace the lower and upper bound constraints~(\ref{eq_p_03}) by a set of single-variable constraints without changing the optimal solution to Problem~\ref{prob}. As these single-variable constraints can be integrated into the existing single-variable constraints~(\ref{eq_p_04}), we can focus without loss of generality on solving Problem~\ref{prob} without this constraint.

To derive this result, we first define for each $j \in \mathcal{M}$ and $S \in \mathbb{R}$ the following subproblem \textlabel{QRAP}{prob_QRA}$^j(S)$ of Problem~\ref{prob}:
\begin{align*}
\text{QRAP}^j(S) \ : \ \min_{x \in \mathbb{R}^{n_j}} \ & \sum_{i \in \mathcal{N}_j} \left( \frac{1}{2} a_i x_i^2 + b_i x_i \right) \\
\text{s.t. } & \sum_{i \in \mathcal{N}_j} x_i = S, \\
& l_i \leq x_i \leq u_i, \quad i \in \mathcal{N}_j.
\end{align*}
Lemma~\ref{lemma_el_new} states the main result of this subsection, namely that optimal solutions to \ref{prob_QRA}$^j(L_j)$ and \ref{prob_QRA}$^j(U_j)$ for $j \in \mathcal{M}$ are component-wise valid lower and upper bounds on optimal solutions to Problem~\ref{prob}. The proof of this lemma is inspired by the proof of Lemma~6.2.1 in \cite{Hochbaum1994} and can be found in Appendix~\ref{sec_app_lemma_el_new}.
\begin{lemma}
For a given $j \in \mathcal{M}$, let $\underline{x}^j := (\underline{x}_i)_{i \in \mathcal{N}}$ and $\bar{x}^j := (\bar{x}_i)_{i \in \mathcal{N}}$ be optimal solutions to \ref{prob_QRA}$^j(L_j)$ and \ref{prob_QRA}$^j(U_j)$ respectively. Then there exists an optimal solution $x^* := (x^*_i)_{i \in \mathcal{N}}$ to Problem~\ref{prob} that satisfies $\underline{x}_i \leq x^*_i \leq \bar{x}_i$ for each $i \in \mathcal{N}_j$.
\label{lemma_el_new}
\end{lemma}

Lemma~\ref{lemma_el_new} implies that adding the inequalities $\underline{x}_i \leq x_i \leq \bar{x}_i$, $i \in \mathcal{N}$ to the formulation of Problem~\ref{prob} does not cut off the optimal solution to the problem. Moreover, these inequalities imply the generalized bound constraints~(\ref{eq_p_03}) since we have for each $j \in \mathcal{M}$ that $\sum_{i \in \mathcal{N}_j} \underline{x}_{i} = L_j$ and $\sum_{i \in \mathcal{N}_j} \bar{x}_{i} = U_j$ by definition of $\underline{x}$ and $\bar{x}$. This means that Problem~\ref{prob} has the same optimal solution as the following problem:
\begin{align*}
\min_{x \in \mathbb{R}^n} \ & \sum_{j=1}^m \frac{1}{2} w_j \left( \sum_{i \in \mathcal{N}_j}  x_i \right)^2 + \sum_{i=1}^n \left(\frac{1}{2}a_i x_i^2 + b_i x_i \right)\\
\text{s.t. } & \sum_{i=1}^n  x_i = R \\
& \underline{x}_i\leq x_i \leq \bar{x}_i, \quad i \in \mathcal{N}.
\end{align*}
To compute the new variable bounds $\underline{x}_i$ and $\bar{x}_i$, we solve the $2m$ subproblems \ref{prob_QRA}$^j(L_j)$ and \ref{prob_QRA}$^j(U_j)$. Since each subproblem is a simple resource allocation problem, this can be done in $O(n)$ time using, e.g., the algorithms in \cite{Kiwiel2008a}. Thus, in the remainder of this article and without loss of generality, we focus on solving Problem~\ref{prob} without Constraint~(\ref{eq_p_03}).

\subsection{Monotonicity of optimal solutions}
\label{sec_KKT}
In this section, we analyze Problem~\ref{prob} (without Constraint~(\ref{eq_p_03})) and the structure of its optimal solutions. More precisely, we study the Karush-Kuhn-Tucker (KKT) conditions (see, e.g., \cite{Boyd2004}) for this problem and derive a property of solutions satisfying all but one of these conditions. This property is the crucial ingredient for our solution approach for Problem~\ref{prob} since it allows us to apply breakpoint search methods for separable convex resource allocation problems.

For convenience, we define $y_j := \sum_{i \in \mathcal{N}_j} x_i$ for $j \in \mathcal{M}$. The KKT-conditions for Problem~\ref{prob} can be written as follows:
\begin{subequations}
\label{KKT_new}%
\begin{flalign}
w_j y_j + a_i x_i + b_i +  \lambda + \mu_i  &= 0, &&j \in \mathcal{M}, \ i \in \mathcal{N}_j  &&\text{(stationarity)} \label{KKT_new_01} \\
\sum_{i=1}^n  x_i &= R && &&\text{(primal feasibility)}\label{KKT_new_02} \\
l_i \leq x_i &\leq u_i, && i \in \mathcal{N} &&\text{(primal feasibility)} \label{KKT_new_03} \\
\mu_i^+ (x_i- u_i) &= 0, && i \in \mathcal{N} &&\text{(complementary slackness)} \label{KKT_new_04}\\
\mu_i^- (x_i - l_i) & = 0, && i \in \mathcal{N} && \text{(complementary slackness)} \label{KKT_new_05} \\
\lambda, \mu_i, & \in \mathbb{R}, && i \in \mathcal{N} && \text{(dual feasibility)} . \label{KKT_new_06}
\end{flalign}
\end{subequations}
Here, $\mu^+_i$ and $\mu^-_i$ are the positive and negative part of $\mu_i$ respectively; i.e., $\mu^+_i= \max(0,\mu_i)$ and $\mu^-_i = \min(0,\mu_i)$. Assuming that Slater's condition holds \cite{Boyd2004}, the KKT-conditions are necessary and sufficient for optimality. Moreover, since Problem~\ref{prob} is \emph{strictly} convex, it has a unique optimal solution~$x^*$.

For a given~$\lambda$, let $(x(\lambda),\mu(\lambda)) \in \mathbb{R}^{2n}$ be the solution that satisfies all KKT-conditions~(\ref{KKT_new}) except~(\ref{KKT_new_02}). Moreover, define $y_j(\lambda):= \sum_{i \in \mathcal{N}_j}  x_i(\lambda)$ for $j \in \mathcal{M}$. It follows that $x(\lambda)$ is the optimal solution to Problem~\ref{prob} if and only if it satisfies KKT-condition~(\ref{KKT_new_02}), i.e., if $\sum_{i=1}^n  x_i(\lambda) = R$. The core of our solution approach is to find a value~$\lambda^*$ such that $\sum_{i=1}^n  x_i(\lambda^*) = R$ and reconstruct the corresponding solution~$x(\lambda^*)$ that, by definition, is optimal to Problem~\ref{prob}. We call~$\lambda^*$ an \emph{optimal (Lagrange) multiplier}.

The main result of this section is Lemma~\ref{lemma_mono}, which states that each $x_i(\lambda)$ can be seen as a non-increasing function of~$\lambda$. This result allows us to use approaches for separable convex resource allocation problems to find~$\lambda^*$. To prove Lemma~\ref{lemma_mono}, we first identify in Lemma~\ref{lemma_tech} a relation between $x_i(\lambda)$ and $\mu_i(\lambda)$.

\begin{lemma}
\label{lemma_tech} %
For any $\lambda_1,\lambda_2 \in \mathbb{R}$ and $i \in \mathcal{N}$, we have that $x_i(\lambda_1) < x_i(\lambda_2)$ implies $\mu_i (\lambda_1) \leq \mu_i (\lambda_2)$.
\end{lemma}
\begin{proof}
Suppose $x_i(\lambda_1) < x_i(\lambda_2)$ for some~$i$. Then $l_i \leq x_i(\lambda_1) < x_i(\lambda_2) \leq u_i$, which implies $x_i(\lambda_1) < u_i$ and $x_i(\lambda_2) > l_i$. Together with KKT-conditions~(\ref{KKT_new_04}) and~(\ref{KKT_new_05}), it follows that $\mu_i(\lambda_1) \leq 0$ and $\mu_i(\lambda_2) \geq 0$ respectively, which implies that $\mu_i(\lambda_1) \leq \mu_i(\lambda_2)$.
\end{proof}

\begin{lemma}
For any $\lambda_1,\lambda_2 \in \mathbb{R}$ such that $\lambda_1 < \lambda_2$, it holds that $x_i(\lambda_1) \geq x_i(\lambda_2)$, $i \in \mathcal{N}$.
\label{lemma_mono}
\end{lemma}
\begin{proof}
See Appendix~\ref{sec_app_lemma_mono}.
\end{proof}
Lemma~\ref{lemma_mono} implies that the values $x_i(\lambda)$ are monotonically decreasing in~$\lambda$. As a consequence, all possible values for the optimal multiplier~$\lambda^*$ form a closed interval $I \subset \mathbb{R}$, i.e., $\lambda \in I$ if and only if $\sum_{i=1}^n x_i(\lambda) = R$. It follows that $\lambda^*$ is non-unique if and only if for each index~$i \in \mathcal{N}$ one of the two bound constraints~\ref{eq_p_04} are tight for~$i$, i.e., either $x^*_i = l_i$ or $x^*_i =u_i$ for all $i \in \mathcal{N}$. Since this constitutes an extreme case and to simplify the discussion, we assume in the derivation of our approach without loss of generality that the optimal multiplier~$\lambda^*$ is unique.

The monotonicity of the values $x_i(\lambda)$ forms the main ingredient for our solution approach to Problem~\ref{prob}, which we derive in Section~\ref{sec_sol}. We conclude this section with two corollaries of Lemma~\ref{lemma_mono} that we require for the derivation of this approach. The first corollary states that not only the values $x_i(\lambda)$ are decreasing in $\lambda$, but also each value $y(\lambda)$. The second corollary is a stronger version of Lemma~\ref{lemma_tech} for the case where $i \in \mathcal{N}_j$ with $w_j < 0$.

\begin{corollary}
For any $\lambda_1,\lambda_2 \in \mathbb{R}$ such that $\lambda_1 < \lambda_2$, it holds that $y_j(\lambda_1) \geq y_j(\lambda_2)$, $j \in \mathcal{M}$.
\label{col_1}
\end{corollary}
\begin{proof}
Follows directly from Lemma~\ref{lemma_mono}.
\end{proof}

\begin{corollary}
If $w_j < 0$, then for any $\lambda_1, \lambda_2 \in \mathbb{R}$ such that $\lambda_1 < \lambda_2$, it holds that $\mu_i(\lambda_1) \geq \mu_i(\lambda_2)$ for $i \in \mathcal{N}_j$.
\label{col_2}
\end{corollary}
\begin{proof}
By Lemma~\ref{lemma_mono}, we have $x_i(\lambda_1) \geq x_i(\lambda_2)$. If this is a strict inequality, i.e., if $x_i(\lambda_1) > x_i(\lambda_2)$, then it follows from Lemma~\ref{lemma_tech} that $\mu_i(\lambda_1) \geq \mu_i(\lambda_2)$. Otherwise, if $x_i(\lambda_1) = x_i(\lambda_2)$, KKT-condition~(\ref{KKT_new_01}) together with $w_j < 0$ and Corollary~\ref{col_1} implies
\begin{align*}
w_j y_j (\lambda_1) + a_i x_i(\lambda_1) + b_i + \mu_i (\lambda_1) &= -\lambda_1 
> -\lambda_2 \\
&= w_j y_j(\lambda_2) +a_i x_i(\lambda_2) + b_i + \mu_i (\lambda_2) \\
& \geq w_j y_j (\lambda_1) + a_i x_i(\lambda_1) + b_i + \mu_i(\lambda_2).
\end{align*}
It follows that $\mu_i (\lambda_1) > \mu_i(\lambda_2)$, proving the corollary.
\end{proof}

\section{Solution approach}
\label{sec_sol}

In this section, we present our approach to solve Problem~\ref{prob}. First, in Section~\ref{sec_outline}, we provide an outline of the approach using the analysis conducted in Section~\ref{sec_analysis}. Second, Section~\ref{sec_comput} focuses in detail on several computational aspects of the approach.

\subsection{Outline}
\label{sec_outline}

The monotonicity of $x_i(\lambda)$, proven in Lemma~\ref{lemma_mono}, has two important implications. First, for each $i \in \mathcal{N}$, there exist unique \emph{breakpoints} $\alpha_i < \beta_i$ such that
\begin{subequations}
\label{eq_bp}
\begin{align}
\lambda \leq \alpha_i &\Leftrightarrow x_i(\lambda) = u_i,  \label{eq_bp_01}\\
\alpha_i < \lambda < \beta_i &\Leftrightarrow l_i < x_i(\lambda) < u_i, \label{eq_bp_02} \\
\beta_i \leq \lambda & \Leftrightarrow x_i(\lambda) = l_i. \label{eq_bp_03}
\end{align}
\end{subequations}
For now, we assume that these breakpoints are known. In Section~\ref{sec_bp}, we discuss how they can be computed efficiently. The second implication of the monotonicity is that, given the optimal multiplier~$\lambda^*$, we have
\begin{subequations}
\label{eq_z}
\begin{align}
\lambda \leq \lambda^* & \Rightarrow \sum_{i=1}^n x_i(\lambda) \geq \sum_{i=1}^n x_i(\lambda^*) = R,  \label{eq_z_01} \\
\lambda \geq \lambda^* & \Rightarrow \sum_{i=1}^n x_i(\lambda) \leq \sum_{i=1}^n x_i(\lambda^*) = R. \label{eq_z_02}
\end{align}
\end{subequations}
These two implications are the base to determine the optimal multiplier~$\lambda^*$. For this, we define the set of all breakpoints by $\mathcal{B} := \lbrace \alpha_i \ | \ i \in \mathcal{N} \rbrace \cup \lbrace \beta_i \ | \ i \in \mathcal{N} \rbrace$. Equations~(\ref{eq_bp_01})-(\ref{eq_bp_03}) imply that $\min (\mathcal{B})  \leq \lambda^* \leq \max(\mathcal{B})$. This means that there exist two consecutive breakpoints $\gamma, \delta \in \mathcal{B}$ such that $\gamma \leq \lambda^* < \delta$. Figure~\ref{fig_example} illustrates the relation between $x(\lambda)$, $y(\lambda)$, the total resource~$R$, the breakpoints in $\mathcal{B}$, and the breakpoints $\gamma$, $\lambda^*$, and $\delta$.
\begin{figure}[ht!]
\caption{Illustrative example of the relation between $x(\lambda)$, $y(\lambda) = \sum_{i=1}^3 x_i(\lambda)$, $R$, and the breakpoints $\alpha_i$, $\beta_i$, $\gamma$, $\lambda^*$, and $\delta$. In this example, $\gamma = \alpha_2$, $\delta = \alpha_3$, and $\lambda^*$ is represented by the black square.}\label{fig_example}
{\includegraphics{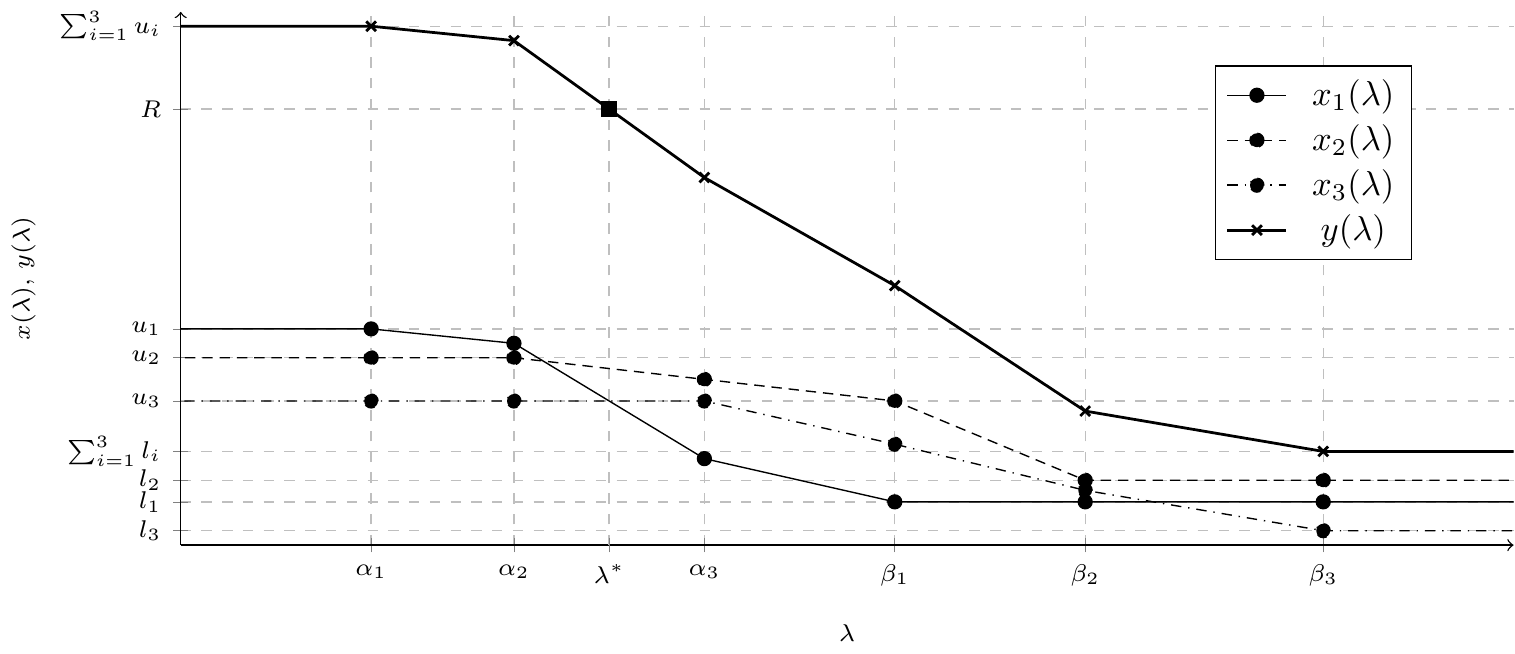}}
\end{figure}

The key of our approach is that once we have found $\gamma$ and $\delta$, we can easily compute $\lambda^*$ and the resulting optimal solution $x(\lambda^*)$. To see this, note that by Equations~(\ref{eq_bp_01})-(\ref{eq_bp_03}) and by definition of $\gamma$, we have for all $i \in \mathcal{N}$ that
\begin{align*}
x_i(\delta) = u_i & \Leftrightarrow x_i(\lambda^*) = u_i, \\
l_i < x_i(\gamma) < u_i & \Leftrightarrow l_i < x_i(\lambda^*) < u_i,  \\
x_i(\gamma) = l_i & \Leftrightarrow x_i(\lambda^*) = l_i. 
\end{align*}
As a consequence, we know that $x_i(\lambda^*) = u_i$ if $\alpha_i \geq \delta$ and $x_i(\lambda^*) = l_i$ if $\beta_i \leq \gamma$. Thus, we may eliminate these variables from the problem. As a consequence for the remaining problem, the box constraints~(\ref{eq_p_04}) become redundant and Problem~\ref{prob} reduces to a quadratic optimization problem with a single equality constraint. We show in Section~\ref{sec_opt_mult} that this specific structure allows us to derive an explicit expression for $\lambda^*$ that can be determined in $O(n)$ time.

To find the breakpoint $\gamma$, we may either consider all breakpoints monotonically in the set~$\mathcal{B}$ of breakpoints or apply a binary search to $\mathcal{B}$. This is because the variable sum $y_i(\lambda) = \sum_{i=1}^n x_i(\lambda)$ induces an order on the breakpoints by Corollary~\ref{col_1}. Moreover, we know by Equations~(\ref{eq_z_01}) and~(\ref{eq_z_02}) that $\gamma$ is the largest breakpoint $\lambda$ in the set $\mathcal{B}$ such that $\sum_{i=1}^n x_i(\lambda) \geq R$. Note, that each of the two approaches to find~$\gamma$ leads to a different algorithm.

The kernel of both approaches is an efficient method to evaluate $x(\lambda)$ for any given~$\lambda \in \mathbb{R}$ as this is the base for computing the breakpoint set $\mathcal{B}$, and to compute $\lambda^*$ from $\gamma$. We focus on each of these three aspects in the next subsection.

\subsection{Computational aspects}
\label{sec_comput}

In the approach outlined in Section~\ref{sec_outline}, there are three quantities whose computation is not straight-forward. These quantities are the solution $x(\lambda)$ for a given $\lambda \in \mathbb{R}$, the set of breakpoints $\mathcal{B}$, and the optimal multiplier $\lambda^*$. In the following subsections, we discuss how these quantities can be computed efficiently.

\subsubsection{Computing $x(\lambda)$ and $y(\lambda)$ for a given $\lambda$}
\label{sec_comp}

To compute $x(\lambda)$ for a given $\lambda$, we need to find a feasible solution to the KKT-conditions~(\ref{KKT_new}) without~(\ref{KKT_new_02}). We call these KKT-conditions the \emph{primary} KKT-conditions. Instead of trying to derive $x(\lambda)$ directly from the primary KKT-conditions, we first determine which variables in $x(\lambda)$ are equal to one of their bounds and which ones are strictly in between their bounds. To this end, for each $j \in \mathcal{M}$, we first partition the set of variables $\mathcal{N}_j$ into the following sets:
\begin{align*}
\mathcal{N}_j^{\text{lower}} (\lambda) &:= \lbrace i \in \mathcal{N}_j \ | \ x_i (\lambda) = l_i \rbrace, \\
\mathcal{N}_j^{\text{upper}} (\lambda)&:= \lbrace i \in \mathcal{N}_j \ | \ x_i (\lambda) = u_i \rbrace, \\
\mathcal{N}_j^{\text{free}} (\lambda)&:= \lbrace i \in \mathcal{N}_j \ | \ l_i < x_i (\lambda) < u_i \rbrace.
\end{align*}
Observe that Equations~(\ref{eq_bp}) imply the following equivalent definition of these sets:
\begin{subequations}
 \label{eq_char_bp}
\begin{align}
\mathcal{N}_j^\text{lower} (\lambda) &= \lbrace i \in \mathcal{N}_j \ | \ \beta_i \leq \lambda \rbrace, \\
\mathcal{N}_j^\text{upper} (\lambda) &= \lbrace i \in \mathcal{N}_j \ | \ \alpha_i \geq \lambda \rbrace, \\
\mathcal{N}_j (\lambda) &= \lbrace i \in \mathcal{N}_j \ | \ \alpha_i < \lambda < \beta_i \rbrace.
\end{align}
\end{subequations}
Thus, given the set $\mathcal{B}$ of breakpoints, we can easily determine these sets in $O(n)$ time by checking whether $\beta_i \leq \lambda$, $\alpha_i \geq \lambda$, or $\alpha_i < \lambda < \beta_i$.

Given the partition $(\mathcal{N}_j^\text{lower} (\lambda),\mathcal{N}_j^\text{upper} (\lambda),\mathcal{N}_j^\text{free} (\lambda))$, we can compute $x(\lambda)$ as follows. As $x_i(\lambda) = l_i$ for all $i \in \mathcal{N}_j^\text{lower} (\lambda)$ and $x_i(\lambda) = u_i$ for all $\mathcal{N}_j^\text{upper} (\lambda)$, it remains to compute $x_i(\lambda)$ for all $i \in \mathcal{N}_j^\text{free} (\lambda)$. Let
\begin{align*}
y^{\text{free}}_j(\lambda) &:= \sum_{i \in \mathcal{N}_j^{\text{free}}(\lambda)} x_i(\lambda), \\
y^{\text{fixed}}_j(\lambda) &:= \sum_{i \in \mathcal{N}_j \backslash \mathcal{N}_j^{\text{free}}(\lambda)} x_i(\lambda)
= \sum_{i \in \mathcal{N}_j^{\text{lower}} (\lambda)} l_i + \sum_{i \in \mathcal{N}_j^{\text{upper}} (\lambda)} u_i.
\end{align*}
By KKT-conditions~(\ref{KKT_new_04})-(\ref{KKT_new_05}), we have $\mu_i(\lambda)= 0$ for each $i \in \mathcal{N}_j^{\text{free}} (\lambda) $. As a consequence, after substituting $y^{\text{fixed}}_j (\lambda)$ and $x_i(\lambda)$ for $i \in \mathcal{N}_j^{\text{lower}} (\lambda) \cup \mathcal{N}_j^{\text{upper}} (\lambda) $ into the primary KKT-conditions~(\ref{KKT_new_01}) and~(\ref{KKT_new_03})-(\ref{KKT_new_05}), the only non-redundant primary KKT-conditions are~(\ref{KKT_new_01}) and~(\ref{KKT_new_06}) for $i \in \mathcal{N}_j^{\text{free}}(\lambda)$:
\begin{equation*}
w_j y_j^{\text{free}}(\lambda) + w_j y_j^{\text{fixed}}(\lambda) + a_i x_i(\lambda) + b_i + \lambda = 0, \quad j \in \mathcal{M}, \ i \in \mathcal{N}_j^{\text{free}}(\lambda), \ \lambda \in \mathbb{R}.
\end{equation*}
We show that the solution to these equations in terms of $x_i(\lambda)$ can be given in closed form. For convenience, we define the following quantities:
\begin{equation*}
A_j(\lambda) := \sum_{\ell \in \mathcal{N}_j^{\text{free}}(\lambda)} \frac{1}{a_{\ell}}, \quad
B_j(\lambda) := \sum_{\ell \in \mathcal{N}_j^{\text{free}}(\lambda)} \frac{b_{\ell}}{a_{\ell}}.
\end{equation*}
Using, e.g., the Sherwood-Morrison formula (see, e.g., \cite{Bartlett1951}), one can deduce and verify that the solution to the non-redundant primary KKT-conditions is
\begin{equation}
x_i(\lambda) = \frac{1}{a_i} \frac{-w_j y_j^{\text{fixed}}(\lambda) - \lambda}{1 + w_j A_j(\lambda)}  - \frac{b_i}{a_i} +\frac{w_j}{a_i} \frac{B_j(\lambda)}{1+w_j A_j(\lambda)},
\quad i \in \mathcal{N}_j^{\text{free}} (\lambda).
\label{eq_x}
\end{equation}
It follows that
\begin{align}
y_j(\lambda) &:= y_j^{\text{free}}(\lambda) + y_j^{\text{fixed}}(\lambda) 
 = -\frac{a_i x_i(\lambda) + b_i + \lambda}{w_j} \nonumber \\
&= \frac{ y_j^{\text{fixed}}(\lambda) + \frac{\lambda}{w_j}}{1 + w_j A_j(\lambda)}
-\frac{B_j(\lambda)}{1+w_j A_j(\lambda)}
- \frac{\lambda}{w_j} 
= \frac{ y_j^{\text{fixed}}(\lambda) - B_j(\lambda)}{1 + w_j A_j(\lambda)}
-\frac{A_j(\lambda)}{1+w_j A_j(\lambda)} \lambda. \label{eq_z_sum}
\end{align}
Note that for each $j \in \mathcal{M}$, $y_j(\lambda)$ can be computed in $O(n_j)$ time given the breakpoint set $\mathcal{B}$. As a consequence, computing the sum $\sum_{j=1}^m y_j(\lambda) (= \sum_{i=1}^n x_i(\lambda))$ takes $O(n)$ time.

\subsubsection{Computing the breakpoints}
\label{sec_bp}

To derive our approach for computing the breakpoints, we exploit two important properties of these breakpoints that we state and prove in Lemmas~\ref{lemma_mu} and~\ref{lemma_xk}. The first property is concerned with the value~$\mu$ introduced in the KKT-conditions~(\ref{KKT_new}). Recall from KKT-conditions~(\ref{KKT_new_04}) and~(\ref{KKT_new_05}) that, for a given $\lambda \in \mathbb{R}$ and $i \in \mathcal{N}$, we have that $\mu_i(\lambda) \geq 0$ if $x_i(\lambda) = u_i$, $\mu_i(\lambda) = 0$ if $l_i < x_i(\lambda) < u_i$, and $\mu_i(\lambda) \leq 0$ if $x_i(\lambda) = l_i$. It follows from Equation~(\ref{eq_bp}) that $\mu_i(\lambda) \geq 0$ if $\lambda \leq \alpha_i$, $\mu_i(\lambda) = 0$ if $\alpha_i < \lambda < \beta_i$, and $\mu_i(\lambda) \leq 0$ if $\beta_i \leq \lambda$. Lemma~\ref{lemma_mu} shows that $\mu_i(\alpha)$ and $\mu_i(\beta_i)$ are equal to the value of $\mu_i(\lambda)$ for $\alpha_i < \lambda < \beta_i$, i.e., are equal to zero.
\begin{lemma}
For all $i \in \mathcal{N}$, we have $\mu_i(\alpha_i) = \mu_i (\beta_i) = 0$.
\label{lemma_mu}
\end{lemma}
\begin{proof}
See Appendix~\ref{sec_app_lemma_mu}.
\end{proof}

Next, Lemma~\ref{lemma_xk} states that, for each $j \in \mathcal{M}$, we can use the values given by $\mathcal{P}_j := \lbrace a_i l_i + b_i \ | \ i \in \mathcal{N}_j \rbrace$ and $\mathcal{Q}_j := \lbrace a_i u_i + b_i \ | \ i \in \mathcal{N}_j \rbrace$ to determine the order of the corresponding breakpoints.
\begin{lemma}
\label{lemma_xk} %
For $j \in \mathcal{M}$ and $i,k \in \mathcal{N}_j$ , we have:
\begin{itemize}
\item
$a_i u_i + b_i > a_k u_k + b_k$ implies $\alpha_i \leq \alpha_k$, and;
\item
$a_i l_i + b_i > a_k l_k + b_k$ implies $\beta_i \leq \beta_k$.
\end{itemize}
\end{lemma}
\begin{proof}
See Appendix~\ref{sec_app_lemma_xk}.
\end{proof}

Lemmas~\ref{lemma_mu} and~\ref{lemma_xk} give rise to the following strategy to compute the breakpoints. From the KKT-condition~(\ref{KKT_new_01}) for $i \in \mathcal{N}_j$, $j \in \mathcal{M}$, we have for the breakpoints $\alpha_i$ and $\beta_i$ that
\begin{subequations}
\label{eq_BP_init}
\begin{align}
\alpha_i &= -w_j y_j(\alpha_i) - a_i x_i (\alpha_i) - b_i - \mu_i(\alpha_i)
= -w_j y_j(\alpha_i) - a_i u_i - b_i - \mu_i(\alpha_i), \label{eq_BP_init_01} \\
\beta_i &= -w_j y_j(\beta_i) - a_i x_i(\beta_i) - b_i - \mu_i(\beta_i)
= -w_j y_j(\beta_i) - a_i l_i - b_i - \mu_i(\beta_i). \label{eq_BP_init_02}
\end{align}
\end{subequations}
Note, that we can obtain the following expression for a given breakpoint $\alpha_i$ by applying Lemma~\ref{lemma_mu} and plugging Equation~(\ref{eq_z_sum}) into Equation~(\ref{eq_BP_init_01}):
\begin{equation}
\label{eq_BP_alpha_expr}
\alpha_i = -w_j y_j(\alpha_i) - a_i u_i - b_i
= \frac{-w_j y_j^{\text{fixed}}(\alpha_i) + w_j B_j(\alpha_i)}{1+ w_j A_j(\alpha_i)} + \frac{w_j A_j(\alpha_i)}{1+w_j A_j(\alpha_i)} \alpha_i - a_i u_i - b_i
\end{equation}
This is equivalent to
\begin{equation}
\label{eq_BP_beta_expr}
\alpha_i = w_j (B_j(\alpha_i) - y_j^{\text{fixed}}(\alpha_i)) - (a_i u_i + b_i)(1+ w_j A_j(\alpha_i)).
\end{equation}
Analogously, we can deduce the following expression for $\beta_i$ by applying Lemma~\ref{lemma_mu} and plugging Equation~(\ref{eq_z_sum}) into Equation~(\ref{eq_BP_init_02}):
\begin{equation*}
\beta_i = w_j (B_j(\beta_i) - y_j^{\text{fixed}}(\beta_i)) - (a_i l_i + b_i)(1+ w_j A_j(\beta_i)).
\end{equation*}
Using these two expressions, we can compute the breakpoints sequentially, i.e., in ascending order. Note that this order can be determined using Lemma~\ref{lemma_xk} without knowledge of the actual values of the breakpoints. For each breakpoint $\eta_k$, we can compute the terms $y_j^{\text{fixed}}(\eta_k)$, $A_j(\eta_k)$, and $B_j(\eta_k)$ efficiently from the preceding breakpoint $\eta_i$ by exploiting the dependencies between the partitions $(\mathcal{N}_j^\text{lower} (\eta_i),\mathcal{N}_j^\text{upper} (\eta_i),\mathcal{N}_j^\text{free} (\eta_i))$ and $(\mathcal{N}_j^\text{lower} (\eta_k),\mathcal{N}_j^\text{upper} (\eta_k),\mathcal{N}_j^\text{free} (\eta_k))$ summarized in Table~\ref{tab_seq} (see also Figure~\ref{fig_example} and Equation~(\ref{eq_char_bp})). Given the smallest breakpoint $\bar{\eta}$, the sequential computation of the terms $y_j^{\text{fixed}}(\cdot)$, $A_j(\cdot)$, and $B_j(\cdot)$ is initialized by $y_j^{\text{fixed}}(\bar{\eta}) := \sum_{i \in \mathcal{N}_j} u_i$, $A_j(\bar{\eta}) := 0$, and $B_j(\bar{\eta}) := 0$. To determine if $\eta_k \equiv \alpha_k$ or $\eta_k \equiv \beta_k$, let $k_1$ be the index of the next lower breakpoint and $k_2$ the index of the next upper breakpoint. Thus, either $\eta_k = \alpha_{k_1}$ or $\eta_k = \beta_{k_2}$. Observe that the partition corresponding to the breakpoint $\eta_k$ does not depend on whether $\eta_k$ is a lower or upper breakpoint. Thus, it follows from the breakpoint expressions in Equations~(\ref{eq_BP_alpha_expr}) and~(\ref{eq_BP_beta_expr}) that $\eta_k = \alpha_{k_1}$ if $a_{k_1} u_{k_1} + b_{k_1} > a_{k_2} l_{k_2} + b_{k_2}$ and $\eta_k = \beta_{k_2}$ otherwise.

\begin{table}[ht!]
\centering
\resizebox{\textwidth}{!}{
\begin{tabular}{r | lll | lll}
\toprule
Type of $\eta_i$ & $\mathcal{N}_j^{\text{lower}}(\eta_k)$ &  $\mathcal{N}_j^{\text{upper}}(\eta_k)$ &  $\mathcal{N}_j^{\text{free}}(\eta_k)$ 
& $y_j^{\text{fixed}}(\eta_k)$ & $A_j(\eta_k)$ & $B_j(\eta_k)$ \\
\midrule
$\eta_i \equiv \alpha_i$ & $\mathcal{N}_j^{\text{lower}}(\eta_i)$ & $\mathcal{N}_j^{\text{upper}}(\eta_i) \backslash \lbrace i \rbrace$
& $\mathcal{N}_j^{\text{free}}(\eta_i) \cup \lbrace i \rbrace$
& $y_j^{\text{fixed}}(\eta_i) - u_i$
& $A_j(\eta_i) + \frac{1}{a_i}$
& $B_j(\eta_i) + \frac{b_i}{a_i}$ \\
$\eta_i \equiv \beta_i$ & $\mathcal{N}_j^{\text{lower}}(\eta_i) \cup \lbrace i \rbrace$ & $\mathcal{N}_j^{\text{upper}}(\eta_i) $
& $\mathcal{N}_j^{\text{free}}(\eta_i) \backslash \lbrace i \rbrace$
& $y_j^{\text{fixed}}(\eta_i) +l_i$
& $A_j(\eta_i) - \frac{1}{a_i}$
& $B_j(\eta_i) - \frac{b_i}{a_i}$ \\
\bottomrule
\end{tabular}
}
\caption{Relation between consecutive breakpoints $\eta_i$ and $\eta_k$ and their index set partitions.}
\label{tab_seq}
\end{table}

Algorithm~\ref{alg_BP} summarizes this approach to compute the breakpoints. Each new smallest breakpoint~$\eta_k$ in Line~4 can be retrieved in $O(1)$ time if we maintain the values in $\mathcal{P}_j$ and $\mathcal{Q}_j$ as sorted lists. As a consequence, the time complexity of Algorithm~\ref{alg_BP} for a given $j \in \mathcal{M}$ is $O(n_j \log n_j)$. Thus, the computation of the breakpoints for all variables can be done in $O(n \log n)$ time. Note that if each $n_j$ is equal to a given constant~$C$, i.e., all subsets $\mathcal{N}_j$ have the same cardinality, this complexity can be refined to $O\left(\sum_{j=1}^m C \log C \right) = O(Cm \log C) = O(n \log C)$. Thus, for a given~$C$ in this case the breakpoints can be computed in linear time.
\begin{algorithm}[ht!]
\caption{Computing the breakpoints for $j \in \mathcal{M}$.}
\label{alg_BP}
\begin{algorithmic}[5]
\STATE{Compute the sets $\mathcal{P}_j := \lbrace a_i l_i + b_i \ | i \in \mathcal{N}_j \rbrace$ and $\mathcal{Q}_j := \lbrace a_i u_i + b_i \ | \ i \in \mathcal{N}_j \rbrace$}
\STATE{Initialize $\bar{Y} := \sum_{i \in \mathcal{N}_j} u_i$; $\bar{A} := 0$; $\bar{B}:=0$}
\REPEAT
\STATE{Take smallest value $\eta_k := \min(\mathcal{P}_j \cup \mathcal{Q}_j)$}
\IF{$\eta_k \in \mathcal{Q}_j$ $\lbrace \eta_k \equiv \alpha_k \rbrace$}
\STATE{$y_j^{\text{fixed}}(\alpha_k) = \bar{Y}$;$A_j(\alpha_k) = \bar{A}$; $B_j(\alpha_k) = \bar{B}$}
\STATE{$\alpha_k := w_j (B_j(\alpha_k) - y_j^{\text{fixed}}(\alpha_k)) - (a_k u_k + b_k)(1 + w_j A_j(\alpha_k))$}
\STATE{$\bar{Y} = \bar{Y} - u_k$; $\bar{A} = \bar{A} + \frac{1}{a_k}$; $\bar{B} = \bar{B} + \frac{b_k}{a_k}$}
\STATE{$\mathcal{Q}_j = \mathcal{Q}_j \backslash \lbrace \eta_k \rbrace$}
\ELSE[$\eta_k \equiv \beta_k$]
\STATE{$y_j^{\text{fixed}}(\beta_k) = \bar{Y}$;$A_j(\beta_k) = \bar{A}$; $B_j(\beta_k) = \bar{B}$}
\STATE{$\beta_k := w_j (B_j(\beta_k) - y_j^{\text{fixed}}(\beta_k)) - (a_k l_k + b_k)(1 + w_j A_j(\beta_k))$}
\STATE{$\bar{Y} = \bar{Y} + l_k$; $\bar{A} = \bar{A} - \frac{1}{a_k}$; $\bar{B} = \bar{B} - \frac{b_k}{a_k}$}
\STATE{$\mathcal{P}_j = \mathcal{P}_j \backslash \lbrace \eta_k \rbrace$}
\ENDIF
\UNTIL{$\mathcal{P}_j \cup \mathcal{Q}_j = \emptyset$}
\end{algorithmic}
\end{algorithm}

\subsubsection{Computing $\lambda^*$}
\label{sec_opt_mult}

To finalize our approach, we need to compute $\lambda^*$ for a given~$\gamma$, which is the largest breakpoint such that $\gamma \leq \lambda^*$. In Section~\ref{sec_outline}, we showed that the partitioning of the variables under $\lambda^*$ can be derived from the partitioning under $\gamma$ and $\delta$, i.e., for each $j \in \mathcal{M}$, we have $\mathcal{N}_j^{\text{lower}}(\lambda^*) = \mathcal{N}_j^{\text{lower}}(\gamma)$, $\mathcal{N}_j^{\text{upper}}(\lambda^*) = \mathcal{N}_j^{\text{upper}}(\delta)$, and $\mathcal{N}_j^{\text{free}}(\lambda^*) = \mathcal{N}_j^{\text{free}}(\gamma)$. Moreover, as we have $\sum_{j=1}^m y_j(\lambda^*) = R$  by definition of $y_j$, we can apply the derived expression for general $y_j(\lambda)$ in Equation~(\ref{eq_z_sum}) to obtain the following linear equation in $\lambda^*$:
\begin{align*}
R = \sum_{j=1}^m y_j(\lambda^*)
&= \sum_{j=1}^m \left( \frac{ y_j^{\text{fixed}}(\lambda^*) -B_j(\lambda^*)}{1 + w_j A_j(\lambda^*)}
-\frac{A_j(\lambda^*)}{1+w_j A_j(\lambda^*)} \lambda^* \right) \\
&
= \sum_{j=1}^m \left( \frac{ y_j^{\text{fixed}}(\gamma) -B_j(\gamma)}{1 + w_j A_j(\gamma)}
-\frac{A_j(\gamma)}{1+w_j A_j(\gamma)} \lambda^* \right).
\end{align*}
It follows that
\begin{equation}
\lambda^* = \frac{\left(
\sum_{j=1}^m \frac{ y_j^{\text{fixed}}(\gamma) -B_j(\gamma)}{1 + w_j A_j(\gamma)}
\right)
-R}
{\sum_{j=1}^m
\frac{A_j(\gamma)}{1+w_j A_j(\gamma)}
}.
\label{eq_opt_mult}
\end{equation}
Note that, given the partitioning sets $\mathcal{N}_j^{\text{lower}}(\lambda^*)$, $\mathcal{N}_j^{\text{upper}}(\lambda^*)$, and $\mathcal{N}_j^{\text{free}}(\lambda^*)$, this expression allows us to compute $\lambda^*$ in $O(\sum_{j=1}^m n_j) = O(n)$ time.

\section{Two algorithms for Problem~\ref{prob}}
\label{sec_alg_two}

In this section, we present two algorithms that solve Problem~\ref{prob} according to the approach presented in Section~\ref{sec_sol}. This approach can be summarized by means of the following four steps:
\begin{enumerate}
\item Replace the generalized bound constraints~(\ref{eq_p_03}) by the box constraints $\underline{x}_i \leq x_i \leq \bar{x}_i$, $i \in \mathcal{N}$ (Section~\ref{sec_replace}),
\item Compute for each $i \in \mathcal{N}$ the lower and upper breakpoints $\alpha_i$ and $\beta_i$ (Section~\ref{sec_bp}),
\item Find $\gamma$ (Section~\ref{sec_outline}),
\item
Compute the optimal Lagrange multiplier~$\lambda^*$ (Section~\ref{sec_opt_mult}) and the optimal solution $x(\lambda^*)$ (Section~\ref{sec_comp}).
\end{enumerate}
Both algorithms follow these four steps. Their difference is in the execution of Step~3 or, more precisely, in how we search for $\gamma$ through the breakpoint set $\mathcal{B}$. In the first algorithm, we consider the breakpoints sequentially starting from the smallest breakpoint, whereas in the second algorithm, we apply binary search on $\mathcal{B}$. We present and discuss these algorithms and their breakpoint search strategies in more detail in Sections~\ref{sec_alg} and~\ref{sec_linear}.

\subsection{An $O(n \log n)$ time algorithm based on sequential breakpoint search}
\label{sec_alg}
The sequential breakpoint search strategy is similar to Algorithm~\ref{alg_BP} to compute the breakpoints, i.e., we search through the breakpoint set $\mathcal{B}$ in ascending order. For each considered breakpoint~$\eta_k$, $k \in \mathcal{N}_j$, we compute the sum $\sum_{j=1}^m y_j (\eta_k)$ using Equation~(\ref{eq_z_sum}). If $\sum_{j=1}^m y_j (\eta_k) > R$, it follows from Equation~(\ref{eq_z_02}) that $\eta_k < \lambda^*$ and we continue the search.  Otherwise, if $\sum_{j=1}^m y_j (\eta_k) < R$, then it follows from Equation~(\ref{eq_z_01}) that $\eta_k > \lambda^*$, meaning that $\gamma$ is the breakpoint preceding $\eta_k$ and that $\delta = \eta_k$. Subsequently, we can use Equation~(\ref{eq_opt_mult}) to compute~$\lambda^*$. Finally, if $\sum_{j=1}^m y_j(\eta_k) = R$, then $\lambda^* = \eta_k$ by definition of the values $y_j(\cdot)$.

To efficiently compute the sum $\sum_{j=1}^m y_j(\eta_k)$, we exploit the dependencies between $\eta_k$ and its preceding breakpoint $\eta_i$, $i \in \mathcal{N}_j$, given in Table~\ref{tab_seq}. This means that we can compute the terms $y_j^{\text{fixed}}(\eta_k)$, $A_j(\eta_k)$, and $B_j(\eta_k)$ in $O(1)$ time from the terms $y_j^{\text{fixed}}(\eta_i)$, $A_j(\eta_i)$, and $B_j(\eta_i)$. Moreover, by defining for a given $\lambda \in \mathbb{R}$
\begin{align*}
F(\lambda) &:= \sum_{j=1}^m \frac{ y_j^{\text{fixed}}(\lambda) - B_j(\lambda)}{1 + w_j A_j(\lambda)}, \\
V(\lambda) &:= \frac{A_j(\lambda)}{1+w_j A_j(\lambda)} .
\end{align*}
and by using these values and the dependencies in Table~\ref{tab_seq}, we can easily compute $\sum_{j=1}^m y_j(\eta_k)$ from $\sum_{j=1}^m y_j(\eta_i)$ in $O(1)$ time. For this, note that $\sum_{j=1}^m y_j(\lambda) = F(\lambda) - \lambda V(\lambda)$ by Equation~(\ref{eq_z_sum}). 

Algorithm~\ref{alg_01} summarizes the four steps of our overall solution approach where Step~3 is carried out using the sequential breakpoint search strategy. In this algorithm, Line~2 corresponds to Step~1, Line~3 to Step~2, Lines~5-36 to Step~3, and Lines~14 and~17 to Step~4. During each iteration~$\tau$ of the sequential breakpoint search in Lines~5-38, the set $\mathcal{B}^{\tau}$ is the part of the original breakpoint set~$\mathcal{B}$ that has not yet been searched in this iteration.

We state the time complexity of this algorithm in the following theorem:
\begin{theorem}
Algorithm~\ref{alg_01} has a worst-case time complexity of $O(n \log n)$.
\label{th_alg_01}
\end{theorem}
\begin{proof}
First, the elimination of Constraint~(\ref{eq_p_03}) in Line~2 takes $O(n)$ time. Second, the computation of the breakpoints by means of Algorithm~\ref{alg_BP} in Line~3 takes $O(n \log n)$ time. Third, each iteration of the sequential breakpoint procedure can be executed in $O(1)$ time if we maintain the breakpoint sets as sorted lists so that computing the smallest value $\eta_k$ in Line~10 can be done in $O(1)$ time. Finally, once $\lambda^*$ has been found in either Line~14 or~17, we can compute the optimal solution $x(\lambda^*)$ in $O(n)$ time using Equation~(\ref{eq_x}). Summarizing, the worst-case time complexity of Algorithm~\ref{alg_01} is $O(n \log n)$.
\end{proof}
Note, that besides the computation and sorting of the breakpoints, Algorithm~\ref{alg_01} runs in linear time.

\begin{algorithm}
\caption{Solving Problem~\ref{prob} using sequential breakpoint search.}
\label{alg_01}
\begin{multicols}{2}
\begin{algorithmic}[5]
\FOR{$j \in \mathcal{M}$}
\STATE{Solve \ref{prob_QRA}$^j(L_j)$ and \ref{prob_QRA}$^j(U_j)$ and set $l_i := \max(l_i, \underline{x}_i(\underline{\lambda}^j(L_j)))$ and $u_i := \min(u_i, \underline{x}_i(\underline{\lambda}^j(U_j)))$}
\STATE{Compute $\alpha_i$ and $\beta_i$ for each $i \in \mathcal{N}_j$ using Algorithm~\ref{alg_BP}}
\ENDFOR
\STATE{$\mathcal{B}  := \lbrace \alpha_i \ | \ i \in \mathcal{N} \rbrace \cup \lbrace \beta_i \ | \ i \in \mathcal{N} \rbrace$; $\tau := 0$; $\mathcal{B}^0 := \mathcal{B}$; $F := \sum_{i=1}^n u_i$; $V :=0$}
\STATE{For $j \in \mathcal{M}$: Initialize $\bar{Y}_j := \sum_{i \in \mathcal{N}_j} u_i$; $\bar{A}_j := 0$; $\bar{B}_j := 0$}
\WHILE{$\lambda^*$ has not been found yet}
\STATE{Take smallest value $\eta_k := \min(\mathcal{B}^{\tau})$ and $j$ with $k \in \mathcal{N}_j$}
\FOR{$j' \in \mathcal{M}$}
\STATE{$y_{j'}^{\text{fixed}}(\eta_k) = \bar{Y}_{j'}$; $A_{j'}(\eta_k) = \bar{A}_{j'}$; $B_{j'}(\eta_k) = \bar{B}_{j'}$}
\ENDFOR
\STATE{Compute $\sum_{j'=1}^m y_{j'}(\eta_k) = F - V \alpha_k$}
\IF{$\sum_{j'=1}^m y_{j'}(\eta_k) = R$}
\STATE{$\lambda^* = \eta_k$; compute $x(\lambda^*)$ as $x(\eta_k)$ using Equation~(\ref{eq_x})}
\RETURN
\ELSIF{$\sum_{j'=1}^m y_{j'}(\eta_k) < R$}
\STATE{$\lambda^* = \frac{F - R}{V}$; compute $x(\lambda^*)$ using Equation~(\ref{eq_x})}
\RETURN
\ELSE
\STATE{$F = F - \frac{\bar{Y}_j - \bar{B}_j}{1 + w_j \bar{A}_j}$}
\STATE{$V = V - \frac{\bar{A}_j}{1 + w_j \bar{A}_j}$}
\IF{$ \eta_k \equiv \alpha_k$}
\STATE{$\bar{Y}_j = \bar{Y}_j - u_i$}
\STATE{$\bar{A}_j = \bar{A}_j + \frac{1}{a_i}$}
\STATE{$\bar{B}_j = \bar{B}_j + \frac{b_i}{a_i}$}
\ELSE
\STATE{$\bar{Y}_j = \bar{Y}_j + l_i$}
\STATE{$\bar{A}_j = \bar{A}_j - \frac{1}{a_i}$}
\STATE{$\bar{B}_j = \bar{B}_j - \frac{b_i}{a_i}$}
\ENDIF
\STATE{$F = F + \frac{\bar{Y}_j - \bar{B}_j}{1 + w_j \bar{A}_j}$}
\STATE{$V = V + \frac{\bar{A}_j}{1 + w_j \bar{A}_j}$}
\STATE{$\mathcal{B}^{\tau + 1} := \mathcal{B}^{\tau} \backslash \lbrace \eta_k \rbrace$}
\STATE{$\tau = \tau + 1$}
\ENDIF
\ENDWHILE
\end{algorithmic}
\end{multicols}
\end{algorithm}

\subsection{An $O(n \log n)$ time algorithm based on binary breakpoint search}
\label{sec_linear}

In this subsection we present an alternative approach, where we apply binary search on the set of breakpoints. During each iteration~$\tau$ of the binary search, we compute the median~$\hat{\gamma}^{\tau}$ of the current breakpoint set $\mathcal{B}^{\tau}$, i.e., of the subset of the original breakpoint set that is guaranteed to contain the breakpoint~$\gamma$. For this median breakpoint, we compute the sum $\sum_{j=1}^m y_j (\hat{\gamma}^{\tau})$ and compare this value to the given amount~$R$ of the resource. If $\sum_{j=1}^m y_j(\hat{\gamma}^{\tau}) = R$, then $\lambda^* = \hat{\gamma}^{\tau}$. Otherwise, if $\sum_{j=1}^m y_j(\hat{\gamma}^{\tau}) < R$, then $\hat{\gamma}^{\tau} > \lambda^* \geq \gamma$ and during the next iteration $\tau + 1$ we take as breakpoint set $\mathcal{B}^{\tau + 1} := \lbrace \lambda \in \mathcal{B}^{\tau} \ | \ \lambda < \hat{\gamma}^{\tau} \rbrace$. Finally, if $\sum_{j=1}^m y_j(\hat{\gamma}^{\tau}) > R$, we have that $\hat{\gamma}^{\tau} < \lambda^* < \delta$ and during the next iteration $\tau + 1$ we take as breakpoint set $\mathcal{B}^{\tau + 1} := \lbrace \lambda \in \mathcal{B}^{\tau} \ | \ \lambda \geq \hat{\gamma}^{\tau} \rbrace$.

To efficiently compute each sum $\sum_{j=1}^m y_j(\hat{\gamma}^{\tau})$, we use the following observation that is inspired by the breakpoint search approach in \cite{Kiwiel2008a} for separable quadratic resource allocation problems. For a given iteration $\tau$ of the binary search, let $\lambda_{\downarrow}^{\tau}$ and $\lambda_{\uparrow}^{\tau}$ denote the minimum and maximum breakpoint in the current breakpoint set $\mathcal{B}^{\tau}$. Then for any multiplier $\lambda$ that lies within the interval $[\lambda_{\downarrow}^{\tau},\lambda_{\uparrow}^{\tau}]$ and each $j \in \mathcal{M}$ and $i \in \mathcal{N}_j$, the following is true due to Equation~(\ref{eq_char_bp}):
\begin{subequations}
\label{eq_lambda_assign}
\begin{align}
\beta_i \leq \lambda_{\downarrow}^{\tau}
& \Rightarrow
i \in \mathcal{N}_j^{\text{lower}}(\lambda), \label{eq_lambda_assign_01} \\
\alpha_i \leq \lambda_{\downarrow}^{\tau} \leq \lambda_{\uparrow}^{\tau} \leq \beta_i
& \Rightarrow
i \in \mathcal{N}_j^{\text{free}}(\lambda), \label{eq_lambda_assign_02} \\
\lambda_{\uparrow}^{\tau} \leq \alpha_i
& \Rightarrow
i \in \mathcal{N}_j^{\text{upper}}(\lambda). \label{eq_lambda_assign_03}
\end{align}
\end{subequations}
We introduce the following sets, which partition the set $\mathcal{N}_j$ according to which of the above cases applies during iteration~$\tau$:
\begin{subequations}
\label{eq_def_book}
\begin{align}
\mathcal{L}_j^{\tau}
&:= \lbrace i \in \mathcal{N}_j \ | \ \beta_i \leq \lambda_{\downarrow}^{\tau} \rbrace, \\
\mathcal{F}_j^{\tau}
&:= \lbrace i \in \mathcal{N}_j \ | \ \alpha_i \leq \lambda_{\downarrow}^{\tau} \leq \lambda_{\uparrow}^{\tau} \leq \beta_i \rbrace, \\
\mathcal{U}_j^{\tau}
&:= \lbrace i \in \mathcal{N}_j \ | \ \lambda_{\uparrow}^{\tau} \leq \alpha_i \rbrace, \\
\mathcal{I}_j^{\tau}
&:= \mathcal{N}_j \backslash (\mathcal{L}_j^{\tau} \cup \mathcal{F}_j^{\tau} \cup \mathcal{U}_j^{\tau})
= \lbrace i \in \mathcal{N}_j \ | \ \lambda_{\downarrow}^{\tau} < \alpha_i < \lambda_{\uparrow}^{\tau} \text{ or } \lambda_{\downarrow}^{\tau} < \beta_i < \lambda_{\uparrow}^{\tau} \rbrace
\end{align}
\end{subequations}
(see also Figure~\ref{fig_bp}). Note, that for any $\lambda$ such that $\lambda_{\downarrow}^{\tau} \leq \lambda \leq \lambda_{\uparrow}^{\tau}$, we have 
\begin{align*}
i \in \mathcal{L}_j^{\tau}
 & \Rightarrow i \in \mathcal{N}_j^{\text{lower}}(\lambda), \\
i \in \mathcal{F}_j^{\tau}
 & \Rightarrow i \in \mathcal{N}_j^{\text{free}}(\lambda), \\
 i \in \mathcal{U}_j^{\tau}
 & \Rightarrow i \in \mathcal{N}_j^{\text{upper}}(\lambda).
 \end{align*}

 \begin{figure}[ht!]
\caption{Example of the partitioning of the variables based on their breakpoints and the interval $[\lambda_{\downarrow}^{\tau},\lambda_{\uparrow}^{\tau}]$.} \label{fig_bp}
{\includegraphics{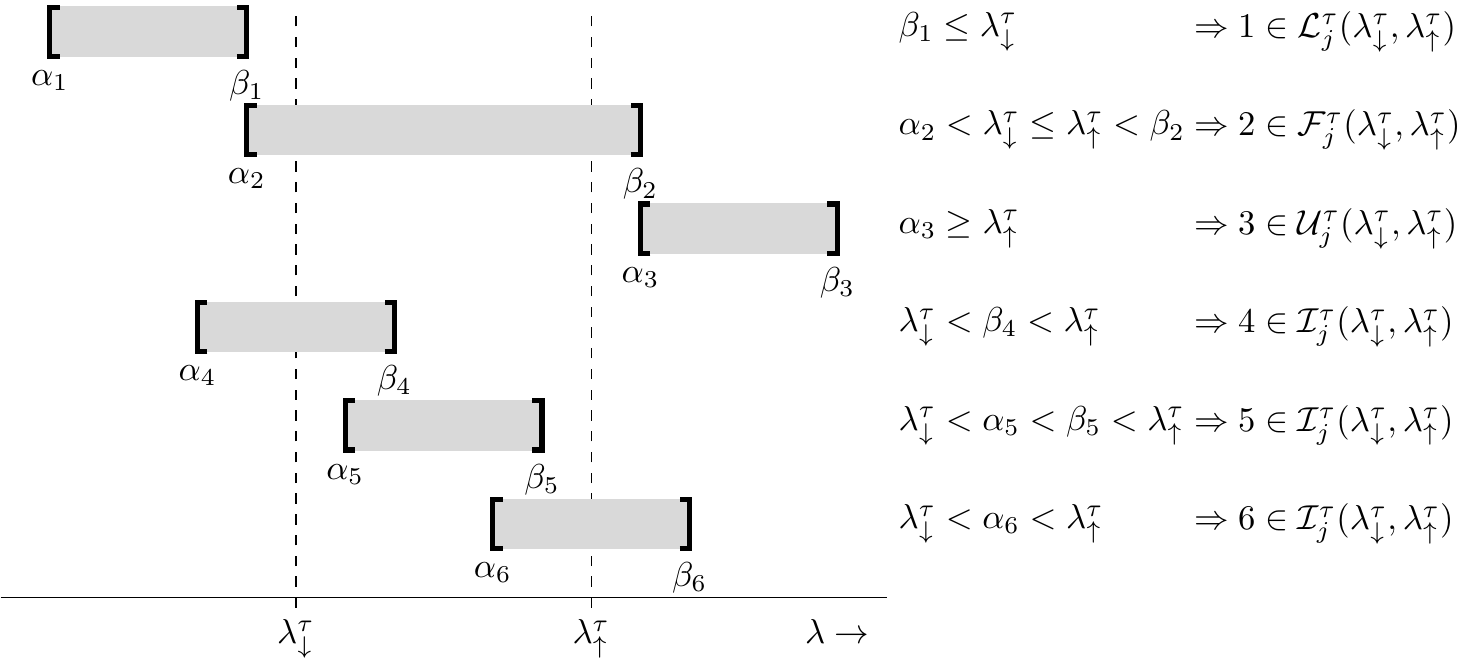}}t
\end{figure}

Due to the construction of the sets $\mathcal{B}^{\tau}$, the sequence $(\lambda_{\downarrow}^{\tau})_{\tau \in \mathbb{N}}$ is nondecreasing and the sequence $(\lambda_{\uparrow}^{\tau})_{\tau \in \mathbb{N}}$ is non-increasing. This implies that as soon as one of the three cases~(\ref{eq_lambda_assign_01})-(\ref{eq_lambda_assign_02}) occurs during an iteration~$\tau$ for an index $i \in \mathcal{N}_j$, we already know for any future candidate breakpoint $\hat{\gamma}^{\bar{\tau}}$ that $i \in \mathcal{N}_j^{\text{lower}}(\hat{\gamma}^{\bar{\tau}})$, $i \in \mathcal{N}_j^{\text{free}}(\hat{\gamma}^{\bar{\tau}})$, or $i \in \mathcal{N}_j^{\text{upper}}(\hat{\gamma}^{\bar{\tau}})$ respectively. In particular, we know that $i \in \mathcal{N}_j^{\text{lower}}(\lambda^*)$, $i \in \mathcal{N}_j^{\text{free}}(\lambda^*)$, or $i \in \mathcal{N}_j^{\text{upper}}(\lambda^*)$ respectively. Thus, when determining the partition $(\mathcal{N}_j^{\text{lower}}(\hat{\gamma}^{\bar{\tau}}), \mathcal{N}_j^{\text{free}}(\hat{\gamma}^{\tau}), \mathcal{N}_j^{\text{upper}}(\hat{\gamma}^{\bar{\tau}}))$, we only need to determine the membership of $x_k(\hat{\gamma}^{\bar{\tau}})$ for all $k \in \mathcal{I}_j^{\tau}$ instead of for all $k \in \mathcal{N}_j$ when the sets $\mathcal{L}_j^{\tau}$, $\mathcal{F}_j^{\tau}$, and $\mathcal{U}_j^{\tau}$ are known.

The main computational gain is obtained by introducing for each iteration~$\tau$ the following bookkeeping parameters:
\begin{equation*}
Y_j^{\tau} := \sum_{i \in \mathcal{L}_j^{\tau}} l_i + \sum_{i \in \mathcal{U}_j^{\tau}} u_i,
\quad
\bar{A}_j^{\tau} := \sum_{i \in \mathcal{F}_j^{\tau}} \frac{1}{a_i},
\quad
\bar{B}_j^{\tau} := \sum_{i \in \mathcal{F}_j^{\tau}} \frac{b_i}{a_i}.
\end{equation*}
Observe that if the set $\mathcal{I}_j^{\tau}$ and the bookkeeping parameters $Y_j^{\tau}$, $\bar{A}_j^{\tau}$, and $\bar{B}_j^{\tau}$ are known, then computing $y_j(\lambda)$ for any $\lambda_{\downarrow}^{\tau} \leq \lambda \leq \lambda_{\uparrow}^{\tau}$ via Equation~(\ref{eq_z}) can be done in $O(|\mathcal{I}_j^{\tau}|)$ time instead of $O(n_j)$ time. 

We summarize the resulting four steps of our overall solution, using in Step 3 the discussed binary breakpoint search strategy, in Algorithm~\ref{alg_02}. In this algorithm, Line~2 corresponds to Step~1, Line~3 to Step~2, Lines~8-46 to Step~3, and Lines~24 and~47-48 to Step~4. In each iteration~$\tau$, the new set $\mathcal{I}_j^{\tau + 1}$ and bookkeeping values $Y_j^{\tau + 1}$, $\bar{A}_j^{\tau + 1}$, and $\bar{B}_j^{\tau + 1}$ are constructed after the new breakpoint set $\mathcal{B}^{\tau + 1}$ and lower and upper bounds $\lambda_{\downarrow}^{\tau + 1}$ and $\lambda_{\uparrow}^{\tau + 1}$ have been determined. This update can be done in line with the definition of the sets $\mathcal{L}_j^{\tau + 1}$, $\mathcal{F}_j^{\tau + 1}$, $\mathcal{U}_j^{\tau + 1}$, and $\mathcal{I}_j^{\tau + 1}$ in Equation~(\ref{eq_def_book}).

We establish the worst-case time complexity of Algorithm~\ref{alg_02} by means of Lemma~\ref{lemma_alg_02_linear} and Theorem~\ref{th_alg_02}. First, Lemma~\ref{lemma_alg_02_linear} states that the binary search procedure can be carried out in $O(n)$ time.
\begin{lemma}
The binary breakpoint search procedure in Lines~8-46 of Algorithm~\ref{alg_02} has a time complexity of $O(n)$.
\label{lemma_alg_02_linear}
\end{lemma}
\begin{proof}
We show that each iteration~$\tau$ of the binary breakpoint search has a time complexity of $O(|\mathcal{B}^{\tau}|)$. Since $|\mathcal{B}^{\tau + 1}| \leq \frac{1}{2} |\mathcal{B}^{\tau}|$ for each iteration $\tau$, it follows that the time complexity of the binary search procedure is
\begin{equation*}
O \left(  \sum_{\tau = 0}^{\log(n)} |\mathcal{B}^{\tau}|  \right) =
O \left( \sum_{\tau = 0}^{\log(n)} \frac{n}{2^{\tau}} \right)
= O(n).
\end{equation*}

We establish the time complexity of an iteration~$\tau$ using the following two observations:
\begin{enumerate}
\item
First, we consider the computation of the candidate multiplier~$\hat{\gamma}^{\tau}$ in Line~9. Note, that the median of an unsorted set of breakpoints~$\mathcal{B}^{\tau}$ can be computed in $O(|\mathcal{B}^{\tau}|)$ time using, e.g., the median-of-medians algorithm in \cite{Blum1973}. This means that instead of sorting the initial breakpoint set~$\mathcal{B}$ in $O(n \log n)$ time and retrieving median elements in $O(1)$ time, we can compute each candidate multiplier~$\hat{\gamma}^{\tau}$ in $O(|\mathcal{B}^{\tau}|)$ time.
\item
Second, by introducing the partition sets $\mathcal{L}_j^{\tau + 1}$, $\mathcal{F}_j^{\tau + 1}$, $\mathcal{U}_j^{\tau + 1}$, and $\mathcal{I}_j^{\tau + 1}$ and the bookkeeping values $Y_j^{\tau}$, $\bar{A}_j^{\tau}$, and $\bar{B}_j^{\tau}$, we reduce the worst-case time complexity of computing $\sum_{j=1}^m y_j(\hat{\gamma}^{\tau})$ from $O(n)$ to $O(\sum_{j=1}^m |\mathcal{I}_j^{\tau}|)$. On the other hand, constructing the new set $\mathcal{I}_j^{\tau + 1}$ and the bookkeeping values $Y_j^{\tau + 1}$, $\bar{A}_j^{\tau + 1}$, and $\bar{B}_j^{\tau + 1}$ in Lines~33-44 takes $O(\sum_{j=1}^m |\mathcal{I}_j^{\tau}|)$ time.
\end{enumerate}
Thus, the time complexity of the $\tau^{\text{th}}$ iteration of the binary search loop is $O(|\mathcal{B}^{\tau}| + \sum_{j=1}^m |I_j^{\tau}|)$. Observe that by definition of $I_j^{\tau}$, for each $j \in \mathcal{M}$ and each index $k \in \mathcal{I}_j^{\tau}$ there is at least one breakpoint ($\alpha_k$ or $\beta_k$ or both) in the set of current breakpoints $\mathcal{B}^{\tau}$. This implies that $\sum_{j=1}^m |I_j^{\tau}| \leq |\mathcal{B}^{\tau}|$. It follows that the time complexity of the $\tau^{\text{th}}$ iteration of the binary search loop reduces to $O(|\mathcal{B}^{\tau}|)$.
\end{proof}

Using this lemma, we establish the time complexity of Algorithm~\ref{alg_02}:
\begin{theorem}
Algorithm~\ref{alg_02} has a time complexity of $O(n \log n)$.
\label{th_alg_02}
\end{theorem}
\begin{proof}
Analogously to Theorem~\ref{th_alg_01}, all operations other than the binary search procedure in Lines~8-46 take $O(n \log n)$ time. Since the binary search procedure takes $O(n)$ time by Lemma~\ref{lemma_alg_02_linear}, the overall time complexity of Algorithm~\ref{alg_02} is $O(n \log n)$.
\end{proof}

\begin{algorithm}
\caption{Solving Problem~\ref{prob} using binary breakpoint search.}
\label{alg_02}
\begin{multicols}{2}
\begin{algorithmic}[5]
\FOR{$j \in \mathcal{M}$}
\STATE{Solve \ref{prob_QRA}$^j(L_j)$ and \ref{prob_QRA}$^j(U_j)$ and set $l_i := \max(l_i, \underline{x}_i(\underline{\lambda}^j(L_j)))$ and $u_i := \min(u_i, \underline{x}_i(\underline{\lambda}^j(U_j)))$}
\STATE{Compute $\alpha_i$ and $\beta_i$ for each $i \in \mathcal{N}_j$ using Algorithm~\ref{alg_BP}}
\ENDFOR
\STATE{$\mathcal{B}  := \lbrace \alpha_i \ | \ i \in \mathcal{N} \rbrace \cup \lbrace \beta_i \ | \ i \in \mathcal{N} \rbrace$; $\mathcal{B}^0 := \mathcal{B}$; $\tau := 0$}
\STATE{For $j \in \mathcal{M}$: $\mathcal{I}_j^0 := \mathcal{N}_j$, $Y_j^0 = \bar{A}_j^0 = \bar{B}_j^0 = 0$}
\STATE{$\lambda_{\downarrow}^0 = -\infty$; $\lambda_{\uparrow}^0 = \infty$}
\WHILE{$|\mathcal{B}^{\tau}| > 1$}
\STATE{$\hat{\gamma}^{\tau} := \text{median}(\mathcal{B}^{\tau})$}
\FOR{$j\in \mathcal{M}$}
\STATE{$y_j^{\text{fixed}}(\hat{\gamma}^{\tau}) := Y_j^{\tau}$; $A_j(\hat{\gamma}) := \bar{A}_j^{\tau}$; $B_j(\hat{\gamma}) := \bar{B}_j^{\tau}$}
\FOR{$k \in \mathcal{I}_j^{\tau}$}
\IF{$k \in \mathcal{N}_j^{\text{lower}}(\hat{\gamma}^{\tau})$}
\STATE{$y_j^{\text{fixed}}(\hat{\gamma}^{\tau}) = y_j^{\text{fixed}}(\hat{\gamma}^{\tau}) + l_k$}
\ELSIF{$k \in \mathcal{N}_j^{\text{upper}}(\hat{\gamma}^{\tau})$}
\STATE{$y_j^{\text{fixed}}(\hat{\gamma}^{\tau}) = y_j^{\text{fixed}}(\hat{\gamma}^{\tau}) + u_k$}
\ELSE
\STATE{$A_j(\hat{\gamma}^{\tau}) = A_j(\hat{\gamma}^{\tau}) + 1 / a_k$; $B_j(\hat{\gamma}^{\tau}) = B_j(\hat{\gamma}^{\tau}) + b_k / a_k$}
\ENDIF
\ENDFOR
\ENDFOR
\STATE{Compute $\sum_{j=1}^m y_j(\hat{\gamma}^{\tau})$ using Equation~(\ref{eq_z})}
\IF{$\sum_{j=1}^m y_j (\hat{\gamma}^{\tau}) = R$}
\STATE{$\lambda^* = \hat{\gamma}^{\tau}$; compute $x(\lambda)$ as $x(\hat{\gamma}^{\tau})$ using Equation~(\ref{eq_x})}
\RETURN
\ELSIF{$\sum_{i=1}^n x_i (\hat{\gamma}^{\tau})< R$}
\STATE{$\mathcal{B}^{\tau + 1} := \lbrace \lambda \in \mathcal{B}^{\tau} \ | \ \lambda < \hat{\gamma}^{\tau} \rbrace$}
\STATE{Determine new bounds: $\lambda_{\downarrow}^{\tau + 1} := \lambda_{\downarrow}^{\tau}$; $\lambda_{\uparrow}^{\tau + 1} := \hat{\gamma}^{\tau}$}
\ELSE
\STATE{$\mathcal{B}^{\tau + 1} := \lbrace \lambda \in \mathcal{B}^{\tau} \ | \ \lambda \geq \hat{\gamma}^{\tau} \rbrace$}
\STATE{Determine new bounds: $\lambda_{\downarrow}^{\tau + 1} := \hat{\gamma}^{\tau}$; $\lambda_{\uparrow}^{\tau + 1} := \lambda_{\uparrow}^{\tau}$}
\ENDIF
\FOR{$j\in \mathcal{M}$}
\STATE{$\mathcal{I}_j^{\tau + 1} := \mathcal{I}_j^{\tau}$; $Y_j^{\tau + 1} := Y_j^{\tau}$; $\bar{A}_j^{\tau + 1}:=\bar{A}_j^{\tau}$; $\bar{B}_j^{\tau + 1} := \bar{B}_j^{\tau}$}
\FOR{$k \in \mathcal{I}_j^{\tau}$}
\IF{$\beta_k \leq \lambda_{\downarrow}^{\tau + 1}$}
\STATE{Remove $k$ from $\mathcal{I}_j^{\tau + 1}$; $Y_j^{\tau + 1} = Y_j^{\tau + 1} + l_k$}
\ELSIF{$\alpha_i \leq \lambda_{\downarrow}^{\tau + 1} \leq \lambda_{\uparrow}^{\tau + 1} \leq \beta_i$}
\STATE{Remove $k$ from $\mathcal{I}_j^{\tau + 1}$; $\bar{A}_j^{\tau + 1} = \bar{A}_j^{\tau + 1} + 1/a_k$; $\bar{B}_j^{\tau + 1} = \bar{B}_j^{\tau + 1} + b_k / a_k$}
\ELSIF{$\lambda_{\uparrow}^{\tau + 1} \leq \alpha_k$}
\STATE{Remove $k$ from $\mathcal{I}_j^{\tau + 1}$; $Y_j^{\tau + 1} = Y_j^{\tau + 1} + u_k$}
\ENDIF
\ENDFOR
\ENDFOR
\STATE{$\tau = \tau + 1$}
\ENDWHILE
\STATE{Determine $\gamma$ as the single element of $\tilde{\mathcal{B}}$}
\STATE{Compute $\lambda^*$ using Equation~(\ref{eq_opt_mult}) and $x(\lambda^*)$ using Equation~(\ref{eq_x})}
\RETURN
\end{algorithmic}
\end{multicols}
\end{algorithm}

\subsection{Complexity results for special cases and related problems}
\label{sec_complexity}

In this section, we use Algorithms~\ref{alg_01} and~\ref{alg_02} and the complexity results in Theorems~\ref{th_alg_01} and~\ref{th_alg_02} to state complexity results for several special cases of Problem~\ref{prob} and related problems. Some of these cases are of interest for the problem of scheduling three-phase electric vehicle charging, whereas other cases may be of independent interest.

The first special case is when all subsets $\mathcal{N}_j$ have the same cardinality, i.e., $|\mathcal{N}_j| = C$ for some natural number $C$. For this case, we can show that, given~$C$, the time complexity of Algorithm~\ref{alg_02} is linear. Note that this special case includes the problem of scheduling three-phase electric vehicle charging that we introduced in Section~\ref{sec_prob_EV} (see also Table~\ref{tab_01}) as we have $C=3$ in this case.
\begin{theorem}
If $n_j = C$ for all $j \in \mathcal{M}$ and $C \in \mathbb{N}$, the time complexity of Algorithm~\ref{alg_02} is $O(n \log C)$.
\end{theorem}
\begin{proof}
The only part of the algorithm that does not have a linear time complexity is the computation of the breakpoints, which needs $O(n \log n)$ operations for the general Problem~\ref{prob}. However, when $n_j = C$, the complexity analysis can be refined to $O(\sum_{j=1}^m n_j \log n_j) = O(\sum_{j=1}^m C \log C) = O(mC \log C) = O(n \log C)$. It follows that the time complexity of Algorithm~\ref{alg_02} for this special case is $O(n \log C)$.
\end{proof}

Next, we focus on the special case where $w_j = 0$ for all $j \in \mathcal{M}$, i.e., the quadratic \emph{separable} resource allocation problem with generalized bound constraints. With regard to three-phase EV charging, this case represents the situation where the only objective is to minimize peak consumption and we do not consider minimization of load unbalance. This case can be solved in $O(n)$ time.
\begin{theorem}
If $w_j = 0$ for all $j \in \mathcal{M}$, Problem~\ref{prob} can be solved in $O(n)$ time.
\label{th_noweight}
\end{theorem}
\begin{proof}
After elimination of the generalized bound constraints~(\ref{eq_p_03}) according to the constraint simplification procedure described in Section~\ref{sec_replace}, the remaining problem is a quadratic separable resource allocation problem since $w_j = 0$ for each $j \in \mathcal{M}$. Thus, we can solve this problem in $O(n)$ time, which implies that we can solve also the whole Problem~\ref{prob} in $O(n)$ time.
\end{proof}

Finally, we consider the integer version of Problem~\ref{prob}, i.e., the problem with the additional constraint that $x_i \in \mathbb{Z}$ for all $i \in \mathcal{N}$. If $w_j = 0$ for all $j \in \mathcal{M}$, we can solve the integer version in $O(n)$ time:
\begin{theorem}
If $w_j = 0$ for all $j \in \mathcal{M}$, the integer version of Problem~\ref{prob} can be solved in $O(n)$ time.
\label{th_int}
\end{theorem}
\begin{proof}
Without loss of generality, we assume that $l,u \in \mathbb{Z}^n$, $L,U \in \mathbb{Z}^m$, and $R \in \mathbb{Z}$. Note that all steps and statements in the proof of Lemma~\ref{lemma_el_new} are valid for the integer version of Problem~\ref{prob} since $\bar{\epsilon} > 1$ and we can choose $\epsilon = 1$ to obtain feasible solutions $x'$ and $(\underline{x}')^j$. Thus, to solve this version, we are required to solve the $2m$ subproblems \ref{prob_QRA}$^j(L_j)$ and \ref{prob_QRA}$^j(U_j)$ and one instance of the quadratic simple resource allocation problem with $n$ variables (see also the proof of Theorem~\ref{th_noweight}) as integer resource allocation problems. Since quadratic simple resource allocation problems with integer variables can be solved in linear time. (see, e.g., \cite[Sections 4.6 and 4.7]{Ibaraki1988}), we can solve these $2m -1$ problems in $O(n)$ time.
\end{proof}

If $w_j \geq 0$ for all $j \in \mathcal{M}$, the integer version can be solved in $O(n^2)$ time \cite{Moriguchi2011}. Finally, we conjecture that the integer version of the general Problem~\ref{prob}, i.e., instances that satisfy Property~\ref{prop_01}, is solvable in strongly polynomial time, but leave this as an open question for future research.

\section{Evaluation}
\label{sec_eval}
In this section, we evaluate the two algorithms presented in Sections~\ref{sec_alg} and~\ref{sec_linear}. We carry out two types of evaluation. First, we evaluate the performance of our algorithms on realistic instances of the EV charging problem \ref{prob_EV} that we introduced in Section~\ref{sec_prob_EV}. Second, to assess the practical scalability of our algorithms, we evaluate them on problem instances with varying numbers~$m$ of generalized bound constraints and numbers~$C$ of variables associated with a given constraint. Since for Problem~\ref{prob} no other tailored algorithms are available, we compare the performance of our algorithms to that of the commercial solver MOSEK \cite{Mosek2019}. 

In Section~\ref{sec_instance}, we describe in more detail the problem instances that we use in the evaluations. Subsequently, in Section~\ref{sec_implement}, we discuss several implementation details. Finally, in Section~\ref{sec_results}, we present and discuss the evaluation results.

\subsection{Problem instances}
\label{sec_instance}

We carry out two types of evaluations. First, we evaluate the performance of our algorithms on instances of Problem~\ref{prob_EV}. For this, we consider a setting wherein an EV is empty and available for residential charging from 18:00h and must be fully charged by 8:00h on the next day. This charging horizon of 14 hours is divided into 15-minute time intervals, meaning that $m = 56$. For the power consumption constraints of the EV, we follow the balancing framework in \cite{Weckx2015} and use the Tesla model~3 as a reference EV \cite{Tesla2020}. This means that we choose $R = 4 \times 40,000 = 160,000$~(Wh), $L_j = 0$~(W) and $U_j = 11,500$~(W) for each $j \in \mathcal{M}$, and $l_i = -\frac{11,500}{3}$~(W) and $u_i = \frac{11,500}{3}$~(W) for each $i \in \mathcal{N}$. We simulate 100 charging sessions, where each session corresponds to charging on a different day. As input for this, we use real power consumption measurement data of 40 households for 100 consecutive days that were obtained in the field test described in \cite{Hoogsteen2017b}. More precisely, we distribute the power consumption profiles of these 40 households randomly over the three phases and, for a given day, choose each parameter $z_{j,p}$ as the sum of the power consumption during interval~$j$ of all households that have been assigned to phase~$p$. To study the influence of different trade-offs between the two objectives (minimizing peak consumption and minimizing load unbalance) on the time required to solve the problem, we simulate each of the 100 charging sessions using three different combinations of the weights~$W_1$ and~$W_2$, namely $(W_1,W_2) \in \lbrace (1,1), (1,100),(100,1) \rbrace$.

Second, we assess the scalability of our algorithms. For this, we focus on the case where the subset sizes $n_j$ are equal to some positive integer~$C$. We generate random instances for a number of fixed values of $C$ and $m$. Table~\ref{tab_scale} shows these fixed values of $C$ and $m$ and for each problem parameter the uniform distribution from which the parameter values are drawn. For each combination of $C$ and $m$, we generate 10 instances according to the given distributions. The distribution of each weight $w_j$ is chosen such that the resulting problem instances satisfy Property~\ref{prop_01}, which ensures by Lemma~\ref{lemma_posdef} that their objective functions are strictly convex. The distributions of the values $L_j$ and $U_j$ are chosen such that none of the generalized bound constraints~(\ref{eq_p_03}) is redundant. As a consequence, all of these constraints need to be replaced according to the constraint simplification procedure in Section~\ref{sec_replace}. Thereby, we maximize the time that Algorithms~\ref{alg_01} and~\ref{alg_02} require for this step and thus improve the fairness of the comparison with MOSEK.

\begin{table}[ht!]
\centering
\begin{tabular}{r | l}
\toprule
Parameter & Values \\
\midrule
$C$ & $\lbrace 1;2;5;10;20;50;100;200;500;1,000 \rbrace$ \\
$m$ & $\lbrace 1;2;5;10;20;50;100;200;500;1,000 \rbrace$ \\
$a_j$ & $\sim U(0,10)$ \\
$b_i$ & $\sim U(-10,10)$ \\
$w_j$ & $\sim U \left( - \frac{1}{\sum_{i \in \mathcal{N}_j} \frac{1}{a_j}},  - \frac{1}{\sum_{i \in \mathcal{N}_j} \frac{1}{a_j}} + 10 \right)$  \\
$l_i$ & $\sim U(-10,0)$ \\
$u_i$ & $\sim U(0,10)$ \\
$L_j$ & $\sim U \left( \sum_{i \in \mathcal{N}_j} l_i , 0.8  \sum_{i \in \mathcal{N}_j} l_i \right)$\\
$U_j$ & $\sim U \left( 0.8 \sum_{i \in \mathcal{N}_j} u_i, \sum_{i \in \mathcal{N}_j} u_i \right)$ \\
$R$ & $\sim U \left( \sum_{j = 1}^m L_j , \sum_{j=1}^m U_j \right)$\\
\bottomrule
\end{tabular}
\caption{Parameter choices for the scalability evaluation.}
\label{tab_scale}
\end{table}

\subsection{Implementation details}
\label{sec_implement}
We implemented our algorithms in Python (version 3.5) to integrate them into DEMKIT, an existing simulation tool for DEM research \cite{Hoogsteen2019}. For solving the subproblems \ref{prob_QRA}$^j(L_j)$ and \ref{prob_QRA}$^j(U_j)$ in Line~2 of both Algorithms~\ref{alg_01} and~\ref{alg_02}, we implement a sophisticated version of the sequential breakpoint search algorithm in \cite{Helgason1980} that allows us to solve both subproblems simultaneously and approximately twice as fast as the original sequential breakpoint search algorithm. Preliminary testing has shown that this algorithm is in general faster than the linear-time algorithms in, e.g., \cite{Kiwiel2008a}, despite its worse time complexity of $O(n_j \log n_j)$. Furthermore, in Algorithm~\ref{alg_02}, we compute the median of a breakpoint set in the same way as in Algorithm~\ref{alg_01}, namely by sorting the original breakpoint set and retrieving the desired breakpoints in $O(1)$ time (see also Section~\ref{sec_alg}). The reason for this is that linear-time algorithms for median finding such as in \cite{Blum1973} are in general slower than alternative sampling-or sorting-based approaches (see, e.g., \cite{Alexandrescu2017}).

In both algorithms, we could reduce the time complexity of sorting all breakpoints from $O(n \log n)$ to $O(n \log m)$ by using a multi-way merging algorithm (see, e.g., \cite{Knuth1998}) to merge the $2m$ sorted lists of breakpoints. However, preliminary testing has shown that in both algorithms the time needed for sorting the breakpoints using a standard sorting algorithm is at least one order of magnitude smaller than the time needed for computing the breakpoints and carrying out the breakpoint search. Thus, we have chosen not to use a multi-way merging algorithm to simplify the implementation of the algorithms without significantly increasing the overall execution time.

\subsection{Results}
\label{sec_results}
In this section we present and discuss the results of the evaluation as described in Section~\ref{sec_instance}. All simulations and computations are executed on a 2.60 GHz Dell Inspiron 15 with an Intel Core i7-6700HQ CPU and 16 GB of RAM.

First, we focus on the performance of our algorithms on the instances of Problem~\ref{prob_EV}. Table~\ref{tab_3phase_results} shows the average execution times of our algorithms and MOSEK for each combination of weights. Moreover, Figure~\ref{plot_3phase_ratio} shows for each combination of weights the boxplots of the ratios between the execution times of each combination of algorithms. The results in Figure~\ref{plot_3phase_ratio} indicate that our algorithms solve realistic instances of Problem~\ref{prob_EV} four to five times as fast as MOSEK for each weight combination. Moreover, Algorithm~\ref{alg_02} appears to be slightly faster than Algorithm~\ref{alg_01}, although the difference in their execution times is less than 4\% of the execution time of Algorithm~\ref{alg_02} for the majority of the problem instances. The results in both Table~\ref{tab_3phase_results} and Figure~\ref{plot_3phase_ratio} suggest that the choice of weights has little to no effect on the execution times of both our algorithms and MOSEK. The results in Table~\ref{tab_3phase_results} indicate that our algorithms can solve realistic instances of Problem~\ref{prob_EV} in the order of milliseconds. This is significantly lower than common speed and delay requirements for communication networks in DEM systems \cite{Deshpande2011}. Thus, our algorithms will most likely not be the (computational) bottleneck in DEM systems and are therefore suitable for integration in such systems. 

\begin{table}[ht!]
\centering
\begin{tabular}{r | rrr}
\toprule
Weight combination & Algorithm~\ref{alg_01} & Algorithm~\ref{alg_02} & MOSEK \\
\midrule
$(1,1)$ & $4.64 \cdot 10^{-3}$ & $4.89 \cdot 10^{-3}$ & $2.33 \cdot 10^{-2}$ \\
$(1,100)$ & $4.69 \cdot 10^{-3}$ & $4.94 \cdot 10^{-3}$ & $2.11 \cdot 10^{-2}$ \\
$(100,1)$ & $4.71 \cdot 10^{-3}$ & $4.86 \cdot 10^{-3}$ & $2.07 \cdot 10^{-2}$ \\
\bottomrule
\end{tabular}
\caption{Average execution times of Algorithms~\ref{alg_01} and~\ref{alg_02} and MOSEK for each combination of weights.}
\label{tab_3phase_results}
\end{table}

\begin{figure}[ht!]
\centering
\caption{Boxplots of the ratios of the execution times between MOSEK and Algorithm~\ref{alg_01} $\left(\frac{\text{MOS.}}{\ref{alg_01}} \right)$, MOSEK and Algorithm~\ref{alg_02} $\left( \frac{\text{MOS.}}{\ref{alg_02}} \right)$, and Algorithms~\ref{alg_02} and~\ref{alg_01} $\left( \frac{\ref{alg_02}}{\ref{alg_01}} \right)$, and MOSEK for each combination of weights.}
\label{plot_3phase_ratio}
\subfloat[$(W_1,W_2) = (1,1)$.]{
\includegraphics{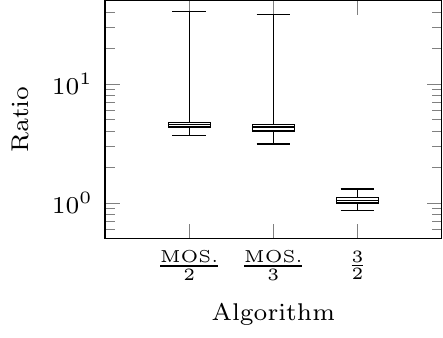}}
\hspace{0.005\textwidth}%
\subfloat[$(W_1,W_2) = (1,100)$.]{
\includegraphics{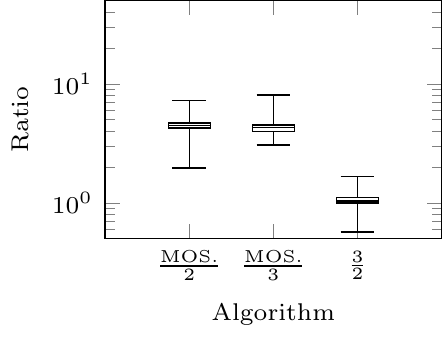}}
\hspace{0.005\textwidth}%
\subfloat[$(W_1,W_2) = (100,1)$.]{
\includegraphics{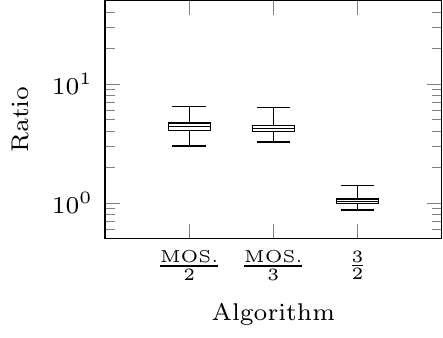}}
\end{figure}

Second, we discuss the results of the scalability evaluation. For this, we first compare the performance of the two different breakpoint search approaches, since this is the only aspect in which Algorithms~\ref{alg_01} and~\ref{alg_02} are different. To this end, we show in Figure~\ref{plot_binary} for each combination of $C$ and $m$ the boxplot of ratios between the execution times of the sequential breakpoint search in Algorithm~\ref{alg_01} and of the binary breakpoint search in Algorithm~\ref{alg_02}, i.e., the execution time of the breakpoint search procedure of Algorithm~\ref{alg_01} divided by that of Algorithm~\ref{alg_02}. Moreover, Figure~\ref{plot_binary_total} shows for each combination of $C$ and $m$ the boxplot of ratios between the overall execution times of Algorithms~\ref{alg_01} and~\ref{alg_02}. The results in Figure~\ref{plot_binary} indicate that for $C \leq 10$ the ratios regarding the breakpoint search procedures decrease significantly as $m$ increases. For these values of~$C$, most of these ratios are greater than 1 when $m \leq 5$ and smaller than 1 when $m \geq 10$. This suggests that the binary breakpoint search procedure is faster than the sequential breakpoint search procedure when $m \geq 10$. For $C > 10$, the relation between the ratios and~$m$ is less clear. However, for each of these values of~$C$, most of the ratios are greater than 1 for almost every value of~$m$, which suggests that the binary breakpoint procedure in general outperforms the sequential breakpoint procedure.

\begin{figure}[ht!]
\centering
\caption{Boxplots of ratios between execution time of the breakpoint search procedures of Algorithms~\ref{alg_01} and~\ref{alg_02}, i.e., the execution time of the breakpoint search procedure of Algorithm~\ref{alg_01} divided by that of Algorithm~\ref{alg_02}}
\label{plot_binary}
\subfloat[$C=1$.]{\includegraphics{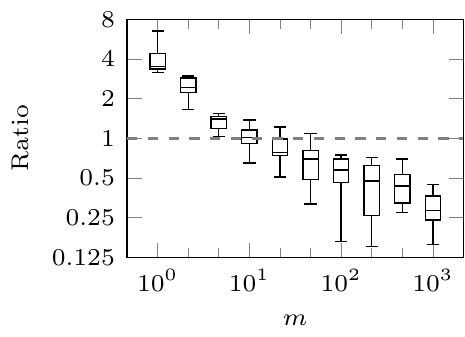}}
\hspace{0.005\textwidth}%
\subfloat[$C=2$.]{\includegraphics{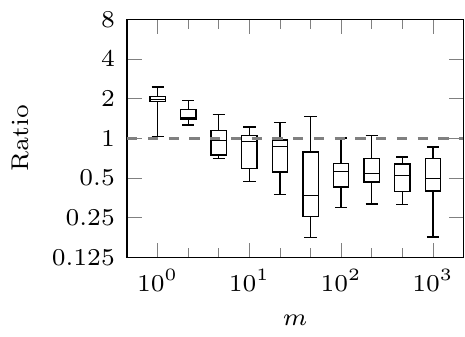}}
\hspace{0.005\textwidth}%
\subfloat[$C=5$.]{\includegraphics{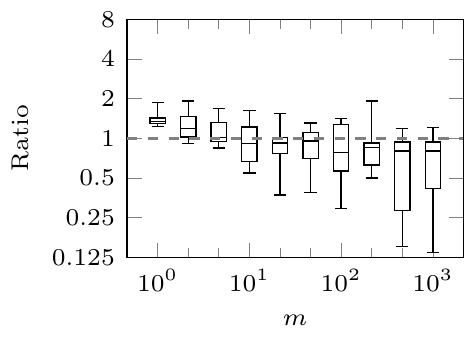}}
\hspace{0.005\textwidth}%

\subfloat[$C=10$.]{\includegraphics{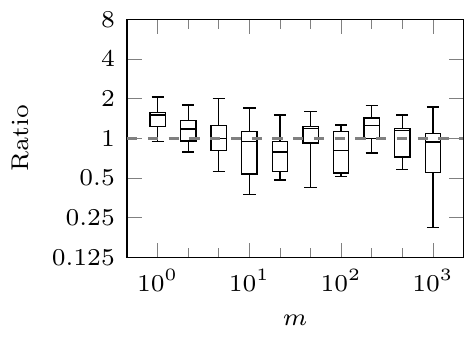}}
\hspace{0.005\textwidth}%
\subfloat[$C=20$.]{\includegraphics{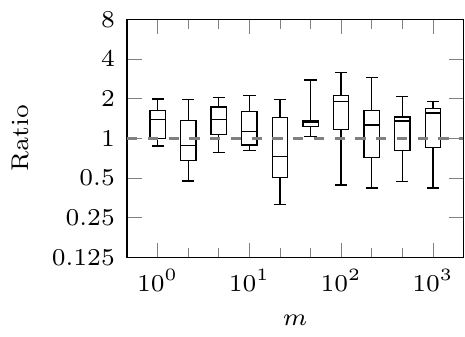}}
\hspace{0.005\textwidth}%
\subfloat[$C=50$.]{\includegraphics{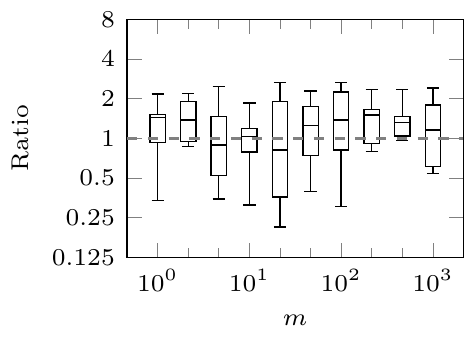}}
\hspace{0.005\textwidth}%

\subfloat[$C=100$.]{\includegraphics{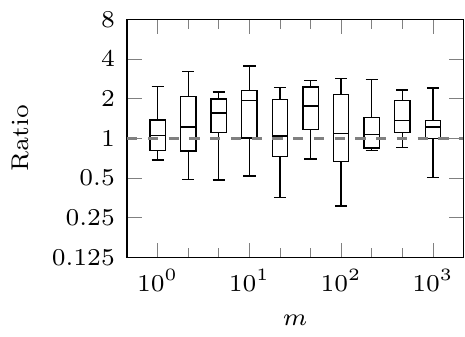}}
\hspace{0.005\textwidth}%
\subfloat[$C=200$.]{\includegraphics{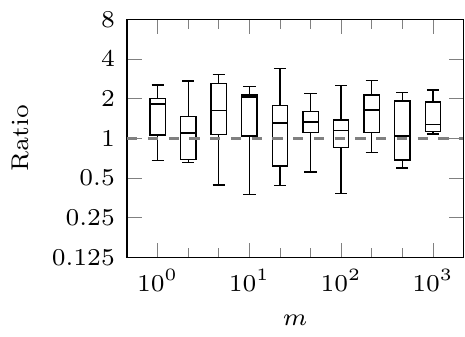}}
\hspace{0.005\textwidth}%
\subfloat[$C=500$.]{\includegraphics{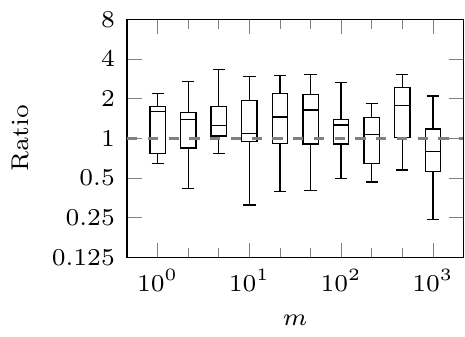}}
\hspace{0.005\textwidth}%

\subfloat[$C=1,000$.]{\includegraphics{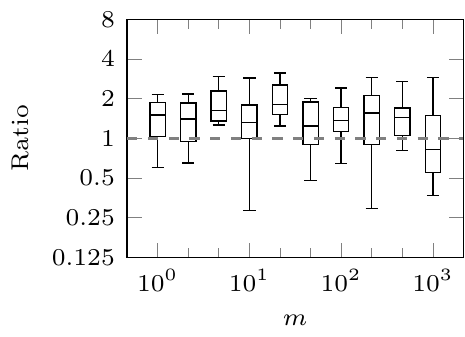}}
\end{figure}

\begin{figure}[ht!]
\centering
\caption{Boxplots of ratios between execution times of Algorithms~\ref{alg_01} and~\ref{alg_02}, i.e., the execution time of Algorithm~\ref{alg_01} divided by that of Algorithm~\ref{alg_02}}
\label{plot_binary_total}
\subfloat[$C=1$.]{\includegraphics{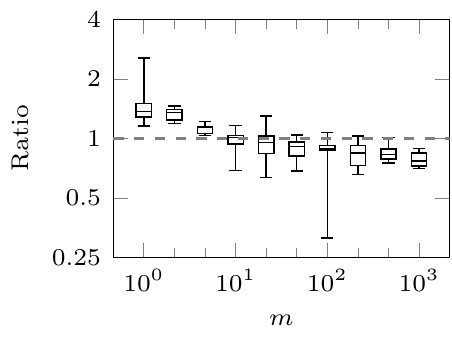}}
\hspace{0.005\textwidth}%
\subfloat[$C=2$.]{\includegraphics{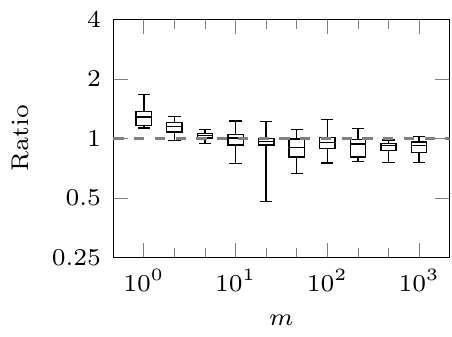}}
\hspace{0.005\textwidth}%
\subfloat[$C=5$.]{\includegraphics{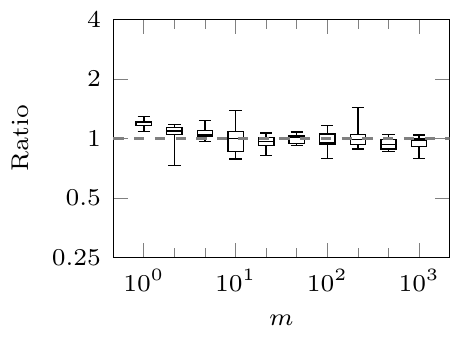}}
\hspace{0.005\textwidth}%

\subfloat[$C=10$.]{\includegraphics{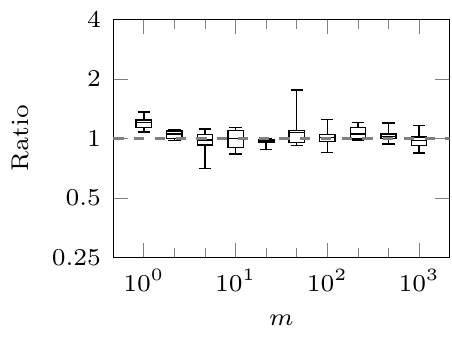}}
\hspace{0.005\textwidth}%
\subfloat[$C=20$.]{\includegraphics{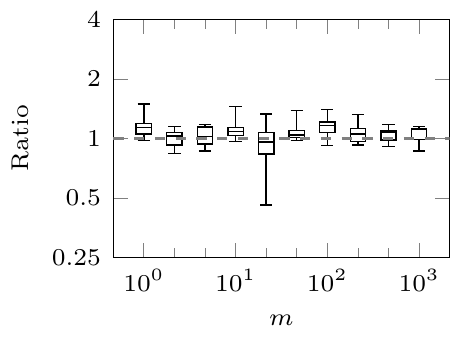}}
\hspace{0.005\textwidth}%
\subfloat[$C=50$.]{\includegraphics{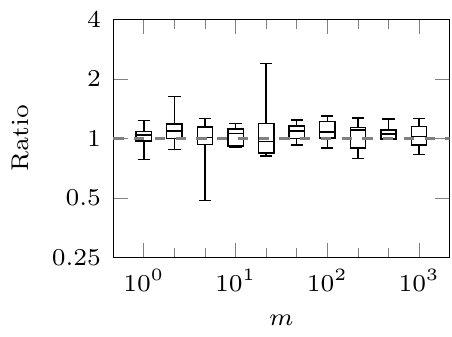}}
\hspace{0.005\textwidth}%

\subfloat[$C=100$.]{\includegraphics{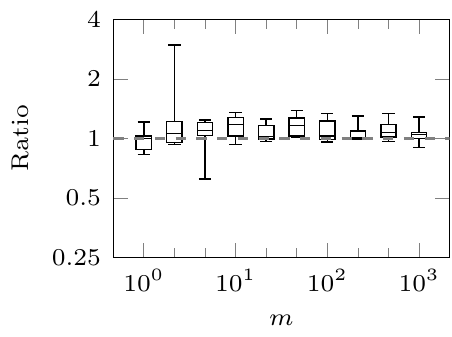}}
\hspace{0.005\textwidth}%
\subfloat[$C=200$.]{\includegraphics{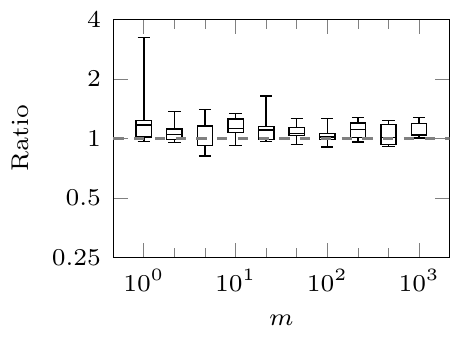}}
\hspace{0.005\textwidth}%
\subfloat[$C=500$.]{\includegraphics{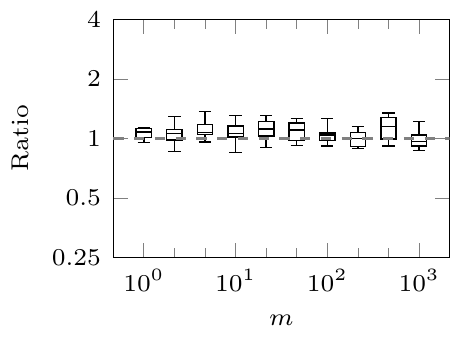}}
\hspace{0.005\textwidth}%

\subfloat[$C=1,000$.]{\includegraphics{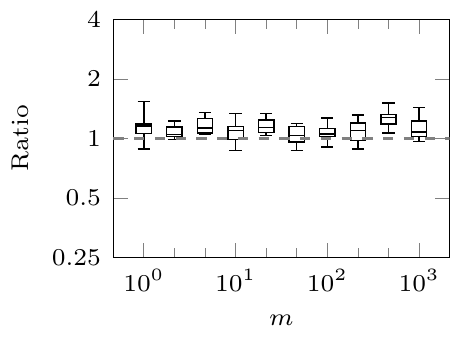}}
\end{figure}

It should be noted that the differences in execution time of the breakpoint searches of Algorithms~\ref{alg_01} and~\ref{alg_02} are less than an order of magnitude. Since the breakpoint search is the only aspect in which the algorithms differ, we expect that the differences in execution time of the entire algorithms are even less. This is confirmed by the results in Figure~\ref{plot_binary_total}, i.e., in almost all cases, the difference in exectuion time between the two algorithms is significantly less than a factor~2. However, the behavior of these ratios is similar to that of those in Figure~\ref{plot_binary}. For example, for $C \leq 10$, most ratios are larger than 1 when $m \leq 5$ and smaller than 1 when $m \geq 10$, whereas for $C > 10$ most ratios are greater than 1. This suggests that Algorithm~\ref{alg_02} is in general faster than Algorithm~\ref{alg_01} unless $C \leq 10$ and $m \leq 5$.

Finally, we compare the performance of our algorithms to MOSEK. To this end, Figure~\ref{plot_avg_compare} shows for each combination of $C$ and $m$ the execution time of Algorithm~\ref{alg_02} and MOSEK on each problem instance. We do not plot the execution times of Algorithm~\ref{alg_01} in this figure, since the differences in execution time between Algorithms~\ref{alg_01} and~\ref{alg_02} are so small that plotting them together in the same figure would unnecessarily obscure the results. Furthermore, Table~\ref{tab_power_law} shows the fitted power laws for Algorithms~\ref{alg_01} and~\ref{alg_02}, i.e., for each $C$, we fit the function $f(m) = c_1 \cdot m^{c_2}$ to the execution times corresponding to $C$. Note that for $C=1,000$, MOSEK was not able to solve any of the instances for $m=500$ and $m=1,000$ due to out-of-memory errors. Finally, to provide additional insight into the reported execution times, Tables~\ref{tab_avg_alg01}-\ref{tab_avg_mosek} in Appendix~\ref{sec_app_results_tab} show for each combination of $C$ and $m$ the average execution time of Algorithms~\ref{alg_01} and~\ref{alg_02} and MOSEK respectively.

The power laws in Figure~\ref{plot_avg_compare} and Table~\ref{tab_power_law} suggest that the execution time of Algorithms~\ref{alg_01} and~\ref{alg_02} grows linearly as~$m$ increases, i.e., the exponents $c_2$ in Table~\ref{tab_power_law} are close to one. This observation is consistent with the theoretical worst-case complexity of Algorithm~\ref{alg_02}, which is $O(n \log C) = O(m C\log C)$ and demonstrates its practical scalability. On the other hand, the execution time of MOSEK does not seem to behave polynomially. Given the execution times of MOSEK for $C \leq 50$ in Figures~\ref{plot_avg_compare}(a)-(f), this is most likely due to the initialization time of MOSEK, which for smaller problem instances is relatively large compared to the actual time required by the internal solver to solve the corresponding instance. As a consequence, Algorithms~\ref{alg_01} and~\ref{alg_02} are at least one order of magnitude faster for instances with $C \leq 50$ and $m \leq 10$.

For $C \geq 100$, the results in Figure~\ref{plot_avg_compare} and Tables~\ref{tab_avg_alg01}-\ref{tab_avg_mosek} indicate that Algorithms~\ref{alg_01} and~\ref{alg_02} are at least one order of magnitude faster than MOSEK regardless of~$m$. In fact, for $C=1,000$, both our algorithms are even two orders of magnitude faster. In this case, our algorithms solve all instances with $m=500$ and $m=1,000$ in less than 16 seconds (Algorithm~\ref{alg_01}) and 12 seconds (Algorithm~\ref{alg_02}), whereas MOSEK was not able to compute a solution due to out-of-memory errors.

\begin{figure}[ht!]
\centering
\caption{Execution times of Algorithm~\ref{alg_02} (circles, black) and MOSEK (triangles, gray).}
\label{plot_avg_compare}
\subfloat[$C=1$.]{\includegraphics{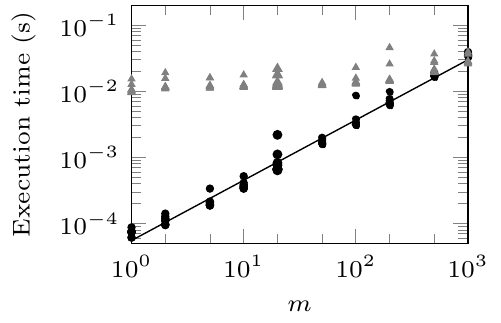}}
\hspace{0.005\textwidth}%
\subfloat[$C=2$.]{\includegraphics{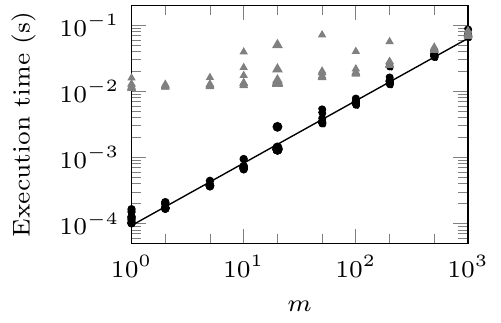}}
\hspace{0.005\textwidth}%
\subfloat[$C=5$.]{\includegraphics{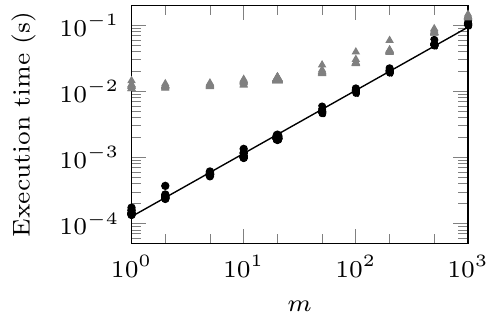}}
\hspace{0.005\textwidth}%

\subfloat[$C=10$.]{\includegraphics{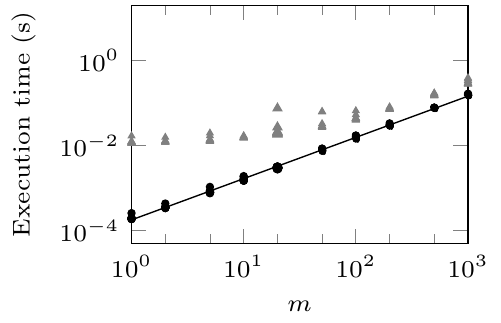}}
\hspace{0.005\textwidth}%
\subfloat[$C=20$.]{\includegraphics{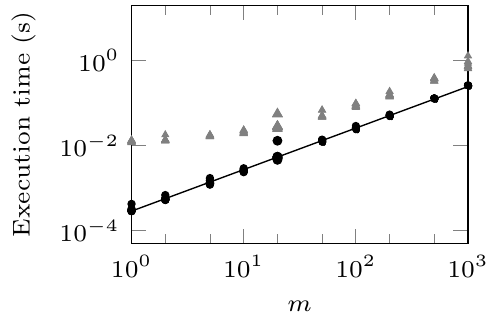}}
\hspace{0.005\textwidth}%
\subfloat[$C=50$.]{\includegraphics{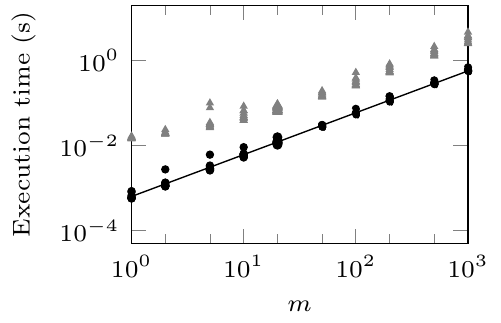}}
\hspace{0.005\textwidth}%

\subfloat[$C=100$.]{\includegraphics{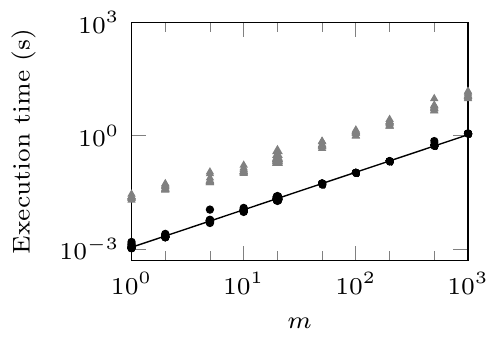}}
\hspace{0.005\textwidth}%
\subfloat[$C=200$.]{\includegraphics{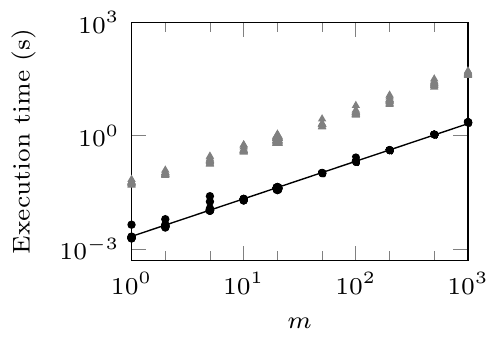}}
\hspace{0.005\textwidth}%
\subfloat[$C=500$.]{\includegraphics{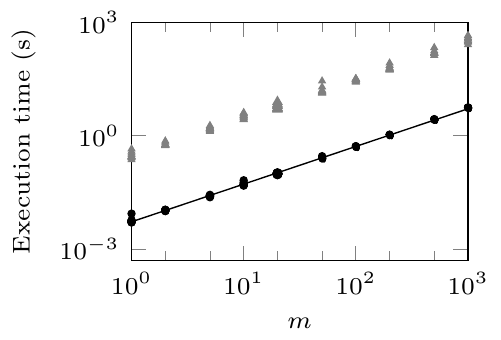}}
\hspace{0.005\textwidth}%

\subfloat[$C=1,000$.]{\includegraphics{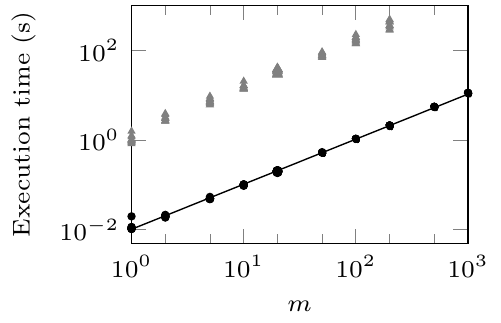}}
\hspace{0.005\textwidth}%
\subfloat[All instances.]{\includegraphics{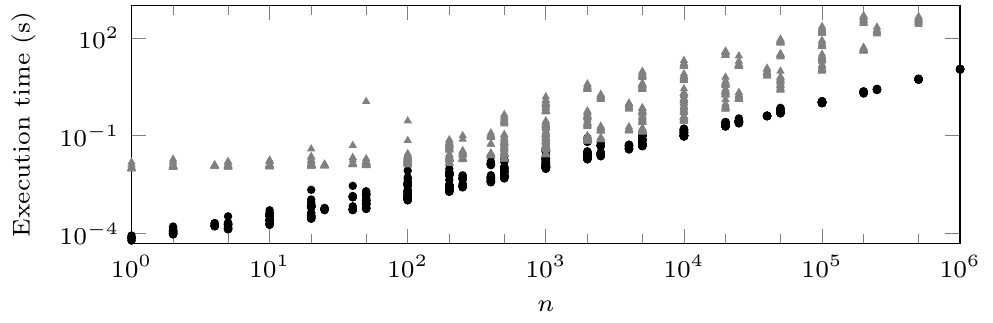}}
\hspace{0.005\textwidth}%
\end{figure}

\begin{table}[ht!]
\centering
\begin{tabular}{r | rr}
\toprule
$C$ & Algorithm~\ref{alg_01} & Algorithm~\ref{alg_02} \\
\midrule
1 & $7.27 \cdot 10^{-5} \cdot m^{0.827}$ & $5.52 \cdot 10^{-5} \cdot m^{0.912}$ \\
2 & $1.08 \cdot 10^{-4} \cdot m^{0.899}$ &	$9.31 \cdot 10^{-5} \cdot m^{0.944}$ \\
5 & $1.40 \cdot 10^{-4} \cdot m^{0.932}$ & $1.27 \cdot 10^{-4} \cdot m^{0.956} $\\
10 & $1.91 \cdot 10^{-4} \cdot m^{0.961}$ & $1.78 \cdot 10^{-4} \cdot m^{0.970}$ \\
20 & $3.06 \cdot 10^{-4} \cdot m^{0.977}$ & $2.87 \cdot 10^{-4} \cdot m^{0.977}$ \\
50 & $6.69 \cdot 10^{-4} \cdot m^{0.987}$ & $6.34 \cdot 10^{-4} \cdot m^{0.985}$ \\
100 & $1.22 \cdot 10^{-3} \cdot m^{0.993}$ & $1.11 \cdot 10^{-3} \cdot m^{0.993}$ \\
200 & $2.50 \cdot 10^{-3} \cdot m^{0.983}$ & $2.15 \cdot 10^{-3} \cdot m^{0.995}$ \\
500 & $5.77 \cdot 10^{-3} \cdot m^{0.991}$ & $5.24 \cdot 10^{-3} \cdot m^{0.998}$ \\
1,000 & $1.14 \cdot 10^{-2} \cdot m^{1.009}$ & $1.03 \cdot 10^{-2} \cdot m^{1.003}$ \\
\bottomrule
\end{tabular}
\caption{Power law regression functions for Algorithms~\ref{alg_01} and~\ref{alg_02} for each~$C$, i.e.,  the fitted functions $c_1 \cdot m^{c_2}$.}
\label{tab_power_law}
\end{table}

\section{Concluding remarks}
\label{sec_concl}

In this article, we studied a quadratic nonseparable resource allocation problem with generalized bound constraints. This problem was motivated by its application in decentralized energy management and in particular for scheduling electric vehicles (EVs) to minimize load unbalance in electricity networks. We derived two algorithms with $O(n\log n)$ time complexity for this problem, of which one runs in linear time for a subclass containing the EV scheduling problem. Numerical evaluations demonstrate the practical efficiency of our algorithms both for realistic instances of the EV scheduling problem and for instances with synthetic data. In fact, our algorithms solve problem instances with up to 1 million variables in less than 16 seconds on a personal computer and are up to 100 times faster than a standard commercial solver. This practical efficiency of our algorithms makes them suitable for the aforementioned electric vehicle scheduling problems since these problems have to be solved on embedded systems with low computational power and low memory.

This work adds a new problem to the class of quadratic nonseparable resource allocation problems that can be solved in strongly polynomial time by efficient algorithms. The question remains how this class can be extended further. Existing work on optimization under submodular constraints \cite{Hochbaum1995,Moriguchi2011} suggests that the class of nonseparable resource allocation problems where both the constraints and nonseparability are induced by a so-called \emph{laminar family} constitutes a promising direction for this extension. We expect that new efficient and practical algorithms can be obtained for these problems by combining insights from existing methodologies to solve similar problems, including minimum quadratic cost flow problems \cite{Tamir1993,Hochbaum1995}, scaling algorithms \cite{Moriguchi2011}, and monotonicity-based optimization (\cite{Vidal2018} and this article). 

With regard to the application of decentralized energy management, these algorithms can be used to solve local optimization problems of devices that are equipped with three-phase chargers other than EVs. In particular, in future research, we focus on the derivation of an algorithm for the (quadratic) nonseparable resource allocation with nested constraints since this models the problem of scheduling large-scale batteries with three-phase chargers. Such batteries are widely recognized as vital components of current and future residential distribution grids with a high infeed from renewable energy sources and integrated devices such as EVs. Therefore, this is a relevant and important direction of future research that can contribute greatly to a sustainable future energy supply.

\section*{Acknowledgments}
This research has been conducted within the SIMPS project (647.002.003) supported by NWO and Eneco.

\appendix

\section{Formulation of Problem~\ref{prob_EV}}
\label{app_obj}

In this appendix, we derive the expression in Equation~(\ref{eq_obj_phase}) for the objective of minimizing load unbalance and show that Problem~\ref{prob_EV} is an instance of Problem~\ref{prob}.

First, as a measure for load unbalance during a given interval~$j \in \mathcal{M}$, we utilize the squared 2-norm of the resulting vector of the three phase loads $q_{j,1}+z_{j,1}$, $q_{j,2} + z_{j,2}$, and $q_{j,3} + z_{j,3}$ according to the phase arrangement depicted in Figure~\ref{fig_phase_01}. This resulting vector equals
\begin{equation*}
P_j^{\text{res}} := \left[ \sum_{p=1}^3 (q_{j,p} + z_{j,p}) \cos \phi_p, \sum_{p=1}^3 (q_{j,p} + z_{j,p}) \sin \phi_p \right],
\end{equation*}
where $\phi_1$, $\phi_2$, and $\phi_3$ are the angles of the three phases. Thus, we can model the objective of minimizing unbalance by minimizing the function $\sum_{j=1}^m ||P_j^{\text{res}}||^2$, where $||\cdot||$ denotes the 2-norm on $\mathbb{R}^2$. Note that we can assume without loss of generality that the phases are arranged as depicted in Fig.~\ref{fig_phase_01}. This means that we may assume that $\phi_1 = 1 \frac{5}{6}\pi$, $\phi_2 = 1\frac{1}{6} \pi$, and $\phi_3 = \frac{1}{2} \pi$. Thus, for each $j= \in \mathcal{M}$, it follows that
\begin{align*}
||P_j^{\text{res}}||^2 &=  \left(\sum_{p=1}^3 (q_{j,p} + z_{j,p}) \cos \phi_p \right)^2 + \left(\sum_{p=1}^3 (q_{j,p} + z_{j,p}) \sin \phi_p \right)^2  \\
&=\sum_{p=1}^3 (q_{j,p} + z_{j,p})^2 (\cos^2 \phi_n + \sin^2 \phi_n) \\
& \quad
+ 2\sum_{p=1}^3 \sum_{p'=p+1}^3 (q_{j,p} + z_{j,p})(q_{j,p'} + z_{j,p'}) (\cos \phi_p \cos \phi_{p'} + \sin \phi_p \sin \phi_{p'})  \\
&= \sum_{p=1}^3 (q_{j,p} + z_{j,p})^2 + 2(q_{j,1} + z_{j,1})(q_{j,2} + z_{j,2}) \left(-\frac{3}{4} +\frac{1}{4} \right) \\
& \quad
+2(q_{j,1} + z_{j,1})(q_{j,3} + z_{j,3}) \left(0 - \frac{1}{2} \right)
+2(q_{j,2} + z_{j,2})(q_{j,3} + z_{j,3})\left( 0 - \frac{1}{2} \right) \\
&= \sum_{p=1}^3 (q_{j,p} + z_{j,p})^2 
- (q_{j,1} + z_{j,1})(q_{j,2} + z_{j,2})
- (q_{j,1} + z_{j,1})(q_{j,3} + z_{j,3})
- (q_{j,2} + z_{j,2})(q_{j,3} + z_{j,3}) \\
&= \frac{3}{2} \sum_{p=1}^3 (q_{j,p} + z_{j,p})^2 -\frac{1}{2} \left(\sum_{p=1}^3 (q_{j,p} + z_{j,p}) \right)^2.
\end{align*}

Second, to show that Problem~\ref{prob_EV} is an instance of Problem~\ref{prob}, observe that the objective function of Problem~\ref{prob_EV} can be rewritten to
\begin{align*}
& W_1 \sum_{j=1}^m \left( \sum_{p=1}^3 (q_{j,p} + z_{j,p}) \right)^2 +  W_2  \sum_{j=1}^m \left( \frac{3}{2} \sum_{p=1}^3 (q_{j,p} + z_{j,p})^2 - \frac{1}{2} \left(\sum_{p=1}^3 (q_{j,p} + z_{j,p}) \right)^2 \right) \\
=& \left(W_1 - \frac{1}{2} W_2 \right) \sum_{j=1}^m \left( \sum_{p=1}^3 (q_{j,p} + z_{j,p}) \right)^2 + \frac{3}{2} W_2  \sum_{j=1}^m  \sum_{p=1}^3 (q_{j,p} + z_{j,p})^2 \\
=& \left(W_1 - \frac{1}{2} W_2 \right) \sum_{j=1}^m \left( \sum_{p=1}^3  z_{j,p} \right)^2 + \frac{3}{2} W_2  \sum_{j=1}^m  \sum_{p=1}^3  z_{j,p}^2 +
\left(W_1 - \frac{1}{2} W_2 \right) \sum_{j=1}^m  \left(\sum_{p=1}^3  q_{j,p}\right) \sum_{p=1}^3 z_{j,p} \\
& \quad 
+ \frac{3}{2} W_2  \sum_{j=1}^m   \sum_{p=1}^3 q_{j,p} z_{j,p}
+ \left(W_1 - \frac{1}{2} W_2 \right) \sum_{j=1}^m \left( \sum_{p=1}^3  q_{j,p} \right)^2 + \frac{3}{2} W_2  \sum_{j=1}^m  \sum_{p=1}^3  q_{j,p}^2
\end{align*}
The latter expression corresponds directly with the values in Table~\ref{tab_01}.

\section{Proofs of Lemmas~\ref{lemma_posdef},~\ref{lemma_el_new}, and~\ref{lemma_mono}-\ref{lemma_xk}}

\subsection{Proof of Lemma~\ref{lemma_posdef}}
\label{sec_app_lemma_posdef}
\begin{lemma_posdef}
$H^j$ is positive definite if and only if $1 + w_j \sum_{i \in \mathcal{N}_j} 1/a_{i} > 0$.
\end{lemma_posdef}
\begin{proof}
Suppose that $H^j$ is positive definite. Then its determinant is strictly positive. Due to the special structure of $H^j$, we can rewrite its determinant to the following form by applying the matrix determinant lemma (see, e.g., \cite{Harville1997}):
\begin{equation*}
\text{det}(H^j) = \text{det}(w_j ee^{\top} + \text{diag}(a^j))
= \left( 1 + w_j \sum_{i \in \mathcal{N}_j} \frac{1}{a_{i}} \right) \text{det}(\text{diag}(a^j)).
\end{equation*}
Since $a_{i} > 0$ for all $i \in \mathcal{N}_j$, we have that $\text{det}(\text{diag}(a^j)) > 0$ and thus we also have that $1 + w_j \sum_{i \in \mathcal{N}_j} 1/a_{i} > 0$.

Now suppose that $1 + w_j \sum_{i \in \mathcal{N}_j} 1/a_{i} > 0$. We show that all leading principal minors of $H^j$ are positive, i.e., that the determinant of each upper-left submatrix of $H^j$ is positive. For this, we label the indices of $\mathcal{N}_j$ as $1,\ldots,n_j$ such that, for any $1 \leq \ell \leq n_j$, the $\ell \times \ell$ upper-left submatrix of $H^j$ is formed by the first $\ell$ rows and columns of $H^j$. Let us denote this submatrix by $H^j_{1:\ell;1:\ell}$.

To show that $\text{det}(H^j_{1:\ell;1:\ell}) > 0$, we compute this determinant by applying the matrix determinant lemma to $H^j_{1:\ell;1:\ell}$. This yields
\begin{equation*}
\text{det}(H^j_{1:\ell;1:\ell}) 
= \left( 1 + w_j \sum_{i =1}^{\ell} \frac{1}{a_{i}} \right) \prod_{i = 1}^{\ell} a_{i}.
\end{equation*}
Note that $1 + w_j \sum_{i =1}^{\ell} 1 / a_{i} > 0$ since $1 + w_j \sum_{i \in \mathcal{N}_j} 1/a_{i} > 0$ and all values~$a_{i}$ are positive. It follows that $\text{det}(H^j_{1:\ell;1:\ell}) > 0$. Since $\ell$ was chosen arbitrarily, this implies that all leading principal minors of $H^j$ are positive and thus that $H^j$ is positive definite.
\end{proof}

\subsection{Proof of Lemma~\ref{lemma_el_new}}
\label{sec_app_lemma_el_new}

\begin{lemma_el_new}
For a given $j \in \mathcal{M}$, let $\underline{x}^j := (\underline{x}_i)_{i \in \mathcal{N}}$ and $\bar{x}^j := (\bar{x}_i)_{i \in \mathcal{N}}$ be optimal solutions to \ref{prob_QRA}$^j(L_j)$ and \ref{prob_QRA}$^j(U_j)$ respectively. Then there exists an optimal solution $x^* := (x^*_i)_{i \in \mathcal{N}}$ to Problem~\ref{prob} that satisfies $\underline{x}_i \leq x^*_i \leq \bar{x}_i$ for each $i \in \mathcal{N}_j$.
\end{lemma_el_new}

\begin{proof}
We prove the validity of the lower bounds $\underline{x}_i \leq x^*_i$; the proof for the upper bounds $x^*_i \leq \bar{x}_i$ is analogous. If for a given $j \in \mathcal{M}$ there is no optimal solution $x^*$ to Problem~\ref{prob} that satisfies the bounds $\underline{x}_i \leq x^*_i \leq \bar{x}_i$ for each $i \in \mathcal{N}_j$, then choose out of all these solutions the one solution $x^*$  for which the value $d := \sum_{\ell \in \mathcal{N}_j} \max(\underline{x}_{\ell} - x^*_{\ell},0)$ is minimum. Let $i \in \mathcal{N}$ be an index with $x^*_i < \underline{x}_i$ and let $j$ be such that $i \in \mathcal{N}_j$. Then there must exist $k \in \mathcal{N}_j \backslash \lbrace i \rbrace$ such that $x^*_k > \underline{x}_k$ since otherwise $\sum_{\ell \in \mathcal{N}_j} x^*_{\ell} < \sum_{\ell \in \mathcal{N}_j} \underline{x}_{\ell} = L_j$. 

Let $\bar{\epsilon} := \min (\underline{x}_i - x^*_i, x^*_k - \underline{x}_k)$ and let $\epsilon \in (0,\bar{\epsilon}]$. Then the solution $x'$ given by
\begin{equation*}
x'_{\ell} = \left\{
\begin{array}{ll}
x^*_{\ell} + \epsilon & \text{if } \ell = i, \\
x^*_{\ell} - \epsilon & \text{if } \ell = k, \\
x^*_{\ell} & \text{otherwise,}
\end{array}
\right.
\end{equation*}
is feasible for Problem~\ref{prob} since $x'_i = x^*_i + \epsilon \leq x^*_i + \bar{\epsilon} \leq x^*_i + \underline{x}_i - x^*_i = \underline{x}_i$, $x'_k = x^*_k - \epsilon \geq x^*_k - \bar{\epsilon} \geq x^*_k - x^*_k + \underline{x}_k = \underline{x}_k$, and $\underline{x}^j$ and $x^*$ are feasible for Problem~\ref{prob_QRA}$^j(L_j)$ and Problem~\ref{prob} respectively. Moreover, since $x^*$ is an optimal solution to Problem~\ref{prob}, we have that
\begin{equation*}
\sum_{j'=1}^m \frac{1}{2} w_{j'} \left( \sum_{\ell \in \mathcal{N}_{j'}}  x^*_i \right)^2 + \sum_{\ell=1}^n \left(\frac{1}{2}a_{\ell} (x^*_{\ell})^2 + b_{\ell} x^*_{\ell} \right) 
\leq
\sum_{j'=1}^m \frac{1}{2} w_{j'} \left( \sum_{\ell \in \mathcal{N}_{j'}}  x'_i \right)^2 + \sum_{\ell=1}^n \left(\frac{1}{2}a_{\ell} (x'_{\ell})^2 + b_{\ell} x'_{\ell} \right) 
.
\end{equation*}
It follows by definition of $x'$ that
\begin{align}
0 &\leq \frac{1}{2} a_i (x'_i)^2 + b_i x'_i + \frac{1}{2} a_k (x'_k)^2 + b_k x'_k 
-
\frac{1}{2} a_i (x^*_i)^2 - b_i x^*_i - \frac{1}{2} a_k (x^*_k)^2 - b_k x^*_k \nonumber
\\
&=\frac{1}{2} a_i (x^*_i + \epsilon)^2 + b_i (x^*_i + \epsilon) + \frac{1}{2} a_k (x^*_k - \epsilon)^2 + b_k (x^*_k - \epsilon) 
-
\frac{1}{2} a_i (x^*_i)^2 - b_i x^*_i - \frac{1}{2} a_k (x^*_k)^2 - b_k x^*_k \nonumber
\\
&=
a_i x^*_i \epsilon + \frac{1}{2}a_i \epsilon^2 + b_i \epsilon
- a_k x^*_k \epsilon + \frac{1}{2}a_k \epsilon^2 + b_k \epsilon. \label{eq_el_new_01}
\end{align}
Analogously, the solution $(\underline{x}')^j$ given by
\begin{equation*}
\underline{x}'_{\ell} = \left\{
\begin{array}{ll}
\underline{x}_{\ell} - \epsilon & \text{if } \ell = i, \\
\underline{x}_{\ell} + \epsilon & \text{if } \ell = k, \\
\underline{x}_{\ell} & \text{otherwise,}
\end{array}
\right.
\end{equation*}
is feasible for \ref{prob_QRA}$^j(L_j)$ since $\underline{x}'_i = \underline{x}_i - \epsilon \geq \underline{x}_i - \bar{\epsilon} \geq \underline{x}_i - \underline{x}_i + x^*_i = x^*_i$, $\underline{x}'_k = \underline{x}_k + \epsilon \leq \underline{x}_k + \bar{\epsilon} \leq \underline{x}_k + x^*_k - \underline{x}_k = x^*_k$, and $x^*$ and $\underline{x}^j$ are feasible for Problem~\ref{prob} and Problem~\ref{prob_QRA}$^j(L_j$ respectively. Moreover, since $\underline{x}^j$ is optimal for Problem~\ref{prob_QRA}$^j(L_j)$, we have that
\begin{equation*}
\sum_{\ell \in \mathcal{N}_j} \left( \frac{1}{2} a_{\ell} (\underline{x}_{\ell})^2 + b_{\ell} \underline{x}_{\ell} \right)
\leq
\sum_{\ell \in \mathcal{N}_j} \left( \frac{1}{2} a_{\ell} (\underline{x}'_{\ell})^2 + b_{\ell} \underline{x}'_{\ell} \right).
\end{equation*}
It follows by definition of $(\underline{x}')^j$ that
\begin{align}
0 &\leq
  \frac{1}{2} a_i (\underline{x}'_{i})^2 + b_i \underline{x}'_i
 +
 \frac{1}{2} a_k (\underline{x}'_{k})^2 + b_k \underline{x}'_k
 -
  \frac{1}{2} a_i (\underline{x}_{i})^2 - b_i \underline{x}_i
 -
 \frac{1}{2} a_k (\underline{x}_{k})^2 - b_k \underline{x}_k  \nonumber \\
 &=
 \frac{1}{2} a_i (\underline{x}_{i} - \epsilon)^2 + b_i (\underline{x}_i - \epsilon)
 +
 \frac{1}{2} a_k (\underline{x}_{k} + \epsilon)^2 + b_k (\underline{x}_k + \epsilon)
 -
  \frac{1}{2} a_i (\underline{x}_{i})^2 - b_i \underline{x}_i
 -
 \frac{1}{2} a_k (\underline{x}_{k})^2 - b_k \underline{x}_k \nonumber \\
 &=
 -a_i \underline{x}_i \epsilon + \frac{1}{2} a_i \epsilon^2 - b_i \epsilon + a_k \underline{x}_k \epsilon + \frac{1}{2} a_k \epsilon^2 + b_k \epsilon.
 \label{eq_el_new_02}
\end{align}
Adding Equations~(\ref{eq_el_new_01}) and~(\ref{eq_el_new_02}) yields
\begin{equation*}
0 \leq a_i \epsilon (x^*_i - \underline{x}_i +\epsilon) + a_k \epsilon (\underline{x}_k - x^*_k + \epsilon)
\leq a_i \epsilon (-\bar{\epsilon} + \epsilon) + a_k \epsilon (-\bar{\epsilon} + \epsilon)
\leq 0.
\end{equation*}
This implies that both Equations~(\ref{eq_el_new_01}) and~(\ref{eq_el_new_02}) are equalities and thus that $x'$ and $(\underline{x}')^j$ are optimal for Problem~\ref{prob} and \ref{prob_QRA}$^j(L_j)$ respectively. However, since $x'_i  \leq \underline{x}_i$ and $x'_k \geq \underline{x}_k$, it holds that
\begin{align*}
\sum_{\ell \in \mathcal{N}_j} \max(\underline{x}_{\ell} - x'_{\ell},0) 
&= d - \max(\underline{x}_i - x^*_i,0) - \max(\underline{x}_k - x^*_k,0) + \max(\underline{x}_i - x'_i,0) + \max(\underline{x}_k - x'_k,0) \\
&= d - \underline{x}_i + x^*_i - 0 + \underline{x}_i - x'_i + 0 \\
&= d + x^*_i - x'_i \\
&= d - \epsilon.
\end{align*}
This is a contradiction with the definition of $x^*$ as the optimal solution that minimizes the expression $\sum_{\ell \in \mathcal{N}_j} \max(\underline{x}_{\ell} - x^*_{\ell},0)$. Hence, there exists an optimal solution satisfying the lower bounds $\underline{x}$.
\end{proof}

\subsection{Proof of Lemma~\ref{lemma_mono}}
\label{sec_app_lemma_mono}

\begin{lemma_mono}
For any $\lambda_1,\lambda_2 \in \mathbb{R}$ such that $\lambda_1 < \lambda_2$, it holds that $x_i(\lambda_1) \geq x_i(\lambda_2)$, $i \in \mathcal{N}$.
\end{lemma_mono}
\begin{proof}
Suppose that there exist $\lambda_1,\lambda_2$ with $\lambda_1 < \lambda_2$ such that for some $j \in \mathcal{M}$ and $i \in \mathcal{N}_j$ we have $x_i(\lambda_1) < x_i(\lambda_2)$. First, we show that there must exist an index $k \in \mathcal{N}_j \backslash \lbrace i \rbrace$ such that $x_k(\lambda_1) \geq x_k(\lambda_2)$. Subsequently, we show that the existence of such an index~$k$ leads to a contradiction.

For each $\ell \in \mathcal{N}_j$, we divide KKT-condition~(\ref{KKT_new_01}) by $a_{\ell}$:
\begin{equation}
\frac{w_j y_j}{a_{\ell}} + x_{\ell} + \frac{b_{\ell} + \lambda + \mu_{\ell}}{a_{\ell}} = 0, \quad \ell \in \mathcal{N}_j.
\label{eq_mono_case2_02}
\end{equation}
By summing Equation~(\ref{eq_mono_case2_02}) over $\mathcal{N}_j$, we obtain
\begin{equation}
0 = 
w_j y_j \sum_{\ell \in \mathcal{N}_j} \frac{1}{a_{\ell}} + \sum_{\ell \in \mathcal{N}_j} \left(x_{\ell} + \frac{b_{\ell} + \lambda + \mu_{\ell}}{a_{\ell}} \right) 
=  \left( 1 + w_j \sum_{\ell \in \mathcal{N}_j} \frac{1}{a_{\ell}} \right)y_j + \sum_{\ell \in \mathcal{N}_j}  \frac{b_{\ell} + \lambda + \mu_{\ell}}{a_{\ell}}.
\label{eq_mono_case2_01}
\end{equation}
Suppose that there is no index $k \in \mathcal{N}_j \backslash \lbrace i \rbrace$ such that $x_k(\lambda_1) \geq x_k(\lambda_2)$. Then for all $\ell \in \mathcal{N}_j$, we have $x_{\ell}(\lambda_1) < x_{\ell}(\lambda_2)$, which in turn implies $y_j (\lambda_1) < y_j (\lambda_2)$. It follows from Equation~(\ref{eq_mono_case2_01}), Property~\ref{prop_01}, and Lemma~\ref{lemma_tech} that
\begin{align*}
0 &= \left( 1 + w_j \sum_{\ell \in \mathcal{N}_j} \frac{1}{a_{\ell}} \right) (y_j (\lambda_2) - y_j(\lambda_1)) + \sum_{\ell \in \mathcal{N}_j} \frac{b_{\ell} - b_{\ell} + \lambda_2 - \lambda_ 1 + \mu_{\ell}(\lambda_2) - \mu_{\ell}(\lambda_1)}{a_{\ell}} \\
& > \sum_{\ell \in \mathcal{N}_j} \frac{b_{\ell} - b_{\ell} + \lambda_2 - \lambda_ 1 + \mu_{\ell}(\lambda_2) - \mu_{\ell}(\lambda_1) }{a_{\ell}}  \\
& > \sum_{\ell \in \mathcal{N}_j} \frac{ \mu_{\ell}(\lambda_2) - \mu_{\ell}(\lambda_1) }{a_{\ell}}  \geq 0.
\end{align*}
This is a contradiction, hence there must exist an index $k \in \mathcal{N}_j \backslash \lbrace i \rbrace$ with $x_k(\lambda_1) > x_k(\lambda_2)$.

We now show that the existence of the index~$k$ leads to a contradiction. By Lemma~\ref{lemma_tech}, we have $\mu_k(\lambda_1) \geq \mu_k(\lambda_2)$. It follows that
\begin{equation}
a_k x_k(\lambda_2) + b_k + \mu_k(\lambda_2)  \leq a_k x_k(\lambda_1) + b_k + \mu_k(\lambda_1).
\label{eq_mono_01}
\end{equation}
However, KKT-condition~(\ref{KKT_new_01}) implies that
\begin{equation*}
w_j y_j(\lambda_1) + a_i x_i (\lambda_1) + b_i + \mu_i(\lambda_1) 
= -\lambda_1 = w_j y_j(\lambda_1) + a_k x_k (\lambda_1) + \mu_k(\lambda_1) ,
\end{equation*}
which yields
\begin{equation}
a_i x_i(\lambda_1) + b_i + \mu_i(\lambda_1) = a_k x_k (\lambda_1) + b_k + \mu_k(\lambda_1) .
\label{eq_mono_03}
\end{equation}
Analogously, we have
\begin{equation}
a_i x_i(\lambda_2) + b_i + \mu_i(\lambda_2)  = a_k x_k(\lambda_2) + b_k + \mu_k(\lambda_2) .
\label{eq_mono_04}
\end{equation}
It follows from Equations~(\ref{eq_mono_01})-(\ref{eq_mono_04}) and Lemma~\ref{lemma_tech} that
\begin{align*}
a_k x_k(\lambda_2) + b_k + \mu_k(\lambda_2)  & \leq a_k x_k(\lambda_1) + b_k + \mu_k(\lambda_1)   \\
& = a_i x_i(\lambda_1) + b_i + \mu_i(\lambda_1)  \\
&< a_i x_i(\lambda_2) + b_i + \mu_i(\lambda_2)  \\
&= a_k x_k(\lambda_2) + b_k + \mu_k(\lambda_2).
\end{align*}
This is a contradiction, hence it must be that $x_i (\lambda_1) \geq x_i(\lambda_2)$. As this implies that $x_i (\lambda_1) \geq x_i(\lambda_2)$ for all $\lambda_1 < \lambda_2$ and $i \in \mathcal{N}$, the lemma is proven.
\end{proof}

\subsection{Proof of Lemma~\ref{lemma_mu}}
\label{sec_app_lemma_mu}

\begin{lemma_mu}
For all $i \in \mathcal{N}$, we have $\mu_i(\alpha_i) = \mu_i (\beta_i) = 0$.
\end{lemma_mu}
\begin{proof}
Let $i \in \mathcal{N}_j$ for some $j \in \mathcal{M}$. We prove the lemma for $\mu_i(\alpha_i)$; the proof for $\mu_i(\beta_i)$ is analogous. Consider the solutions $x(\alpha_i)$ and $x(\alpha_i + \epsilon)$ for an arbitrary $\epsilon > 0$ with $\alpha_i + \epsilon < \beta_i$. Note that $\mu_i(\alpha_i + \epsilon) = 0$ by Equation~(\ref{eq_bp_02}) and KKT-conditions~(\ref{KKT_new_04}) and~(\ref{KKT_new_05}). It follows from KKT-condition (\ref{KKT_new_01}) that
\begin{subequations}
\begin{align}
w_j y_j(\alpha_i) + a_{\ell} x_{\ell}(\alpha_i) + b_{\ell} + \alpha_i + \mu_{\ell}(\alpha_i) &=0, \quad \ell \in \mathcal{N}_j,\label{eq_lemma_mu_1}\\
w_j y_j (\alpha_i + \epsilon) + a_{\ell} x_{\ell}(\alpha_i + \epsilon) + b_{\ell} + \alpha_i + \epsilon + \mu_{\ell}(\alpha_i + \epsilon) &=0, \quad \ell \in \mathcal{N}_j. \label{eq_lemma_mu_2}
\end{align} \label{eq_lemma_mu}%
\end{subequations}
To show that $\mu_i(\alpha_i) = 0$, we show that $\mu_i(\alpha_i) \leq \epsilon$ if $w_j \geq 0$ and $\mu_i(\alpha_i) \leq \epsilon \sum_{\ell \in \mathcal{N}_j} \frac{1}{a_{\ell}}$ if $w_j < 0$. Since $\epsilon$ was chosen arbitrarily, $\sum_{\ell \in \mathcal{N}_j} \frac{1}{a_{\ell}} > 0$, and $\mu_i(\alpha_i) \geq 0$ by definition of $\alpha_i$, this implies in both cases that $\mu_i(\alpha_i) = 0$.

First, if $w_j \geq 0$, then $w_j y_j(\cdot)$ is non-increasing by Corollary~\ref{col_1}. Together with Lemma~\ref{lemma_mono} and the fact that $\mu_i(\alpha_i + \epsilon) = 0$, subtracting Equation~(\ref{eq_lemma_mu_2}) from Equation~(\ref{eq_lemma_mu_1}) for $\ell = i$ yields
\begin{align*}
0 &= w_j y_j(\alpha_i) - w_j y_j(\alpha_i + \epsilon) + a_i x_i(\alpha_i) - a_i x_i(\alpha_i + \epsilon) - \epsilon + \mu_i(\alpha_i) - \mu_i(\alpha_i + \epsilon) \\
&\geq - \epsilon + \mu_i(\alpha_i).
\end{align*}
It follows that $\mu_i(\alpha_i)  \leq \epsilon$. 

Second, if $w_j < 0$, then we can apply the same proof mechanism as was used in the proof of Lemma~\ref{lemma_mono} (see also Equations~(\ref{eq_mono_case2_02}) and~(\ref{eq_mono_case2_01})). By dividing Equations~(\ref{eq_lemma_mu_1}) and~(\ref{eq_lemma_mu_2}) by $a_{\ell}$ and summing them over the index set $\mathcal{N}_j$, we get the following together with Property~\ref{prop_01} and Corollary~\ref{col_2}:
\begin{align*}
0 &= \sum_{\ell \in \mathcal{N}_j} \left( \frac{w_j y_j(\alpha_i)}{a_{\ell}} - \frac{w_j y_j(\alpha_i + \epsilon)}{a_{\ell}} +  x_{\ell}(\alpha_i) -  x_{\ell}(\alpha_i + \epsilon) -\frac{ \epsilon}{a_{\ell}} + \frac{\mu_{\ell}(\alpha_i)}{a_{\ell}} - \frac{\mu_{\ell}(\alpha_i + \epsilon)}{a_{\ell}} \right) \\
&= \left(1 + w_j\sum_{\ell \in \mathcal{N}} \frac{1}{a_{\ell}} \right) (y_j(\alpha_i) - y_j(\alpha_i + \epsilon))
- \epsilon \sum_{\ell \in \mathcal{N}_j} \frac{1}{a_{\ell}}
+ \sum_{\ell \in \mathcal{N}_j} \frac{\mu_{\ell}(\alpha_i) - \mu_{\ell}(\alpha_i + \epsilon)}{a_{\ell}} \\
& \geq 
- \epsilon \sum_{\ell \in \mathcal{N}_j} \frac{1}{a_{\ell}}
+ \sum_{\ell \in \mathcal{N}_j} \frac{\mu_{\ell}(\alpha_i) - \mu_{\ell}(\alpha_i + \epsilon)}{a_{\ell}} \\
& \geq  - \epsilon \sum_{\ell \in \mathcal{N}_j} \frac{1}{a_{\ell}}
+ \mu_i (\alpha_i).
\end{align*}
Here, the first inequality follows from Property~\ref{prop_01} and Lemma~\ref{lemma_mono} and the second equality follows from Corollary~\ref{col_2} and the fact that $\mu_i(\alpha_i + \epsilon) = 0$. It follows that $\mu_i(\alpha_i) \leq \epsilon  \sum_{\ell \in \mathcal{N}_j} \frac{1}{a_{\ell}}$.
\end{proof}

\subsection{Proof of Lemma~\ref{lemma_xk}}
\label{sec_app_lemma_xk}

\begin{lemma_xk}
For $j \in \mathcal{M}$ and $i,k \in \mathcal{N}_j$ , we have:
\begin{itemize}
\item
$a_i u_i + b_i > a_k u_k + b_k$ implies $\alpha_i \leq \alpha_k$, and;
\item
$a_i l_i + b_i > a_k l_k + b_k$ implies $\beta_i \leq \beta_k$.
\end{itemize}
\end{lemma_xk}
\begin{proof}
We prove the lemma for the case $a_i u_i + b_i > a_k u_k + b_k$; the proof for the case $a_i l_i + b_i > a_k l_k + b_k$ is analogous. We show that $x_i(\alpha_k) < u_i$, which implies by definition of $\alpha_i$ that $x_i(\alpha_k) < u_i = x_i(\alpha_i)$. Using Lemma~\ref{lemma_mono}, this yields $\alpha_i \leq \alpha_k$.

It follows from KKT-condition~(\ref{KKT_new_01}) that
\begin{equation*}
w_j y_j(\alpha_k) + a_k x_k(\alpha_k) + b_k + \mu_k(\alpha_k) = -\alpha_k = w_j y_j(\alpha_k) + a_i x_i(\alpha_k) + b_i + \mu_i(\alpha_k).
\end{equation*}
Since $\mu_k(\alpha_k) = 0$ by Lemma~\ref{lemma_mu} and $x_k(\alpha_k) = u_k$ by definition of~$\alpha_k$, the above is equivalent to
\begin{equation}
a_k u_k + b_k = a_i x_i(\alpha_k) + b_i + \mu_i(\alpha_k).
\label{eq_xi_01}
\end{equation}
Suppose that $x_i(\alpha_k) = u_i$. Then $\mu_i(\alpha_k) \geq 0$ by KKT-condition~(\ref{KKT_new_04}). It follows from Equation~(\ref{eq_xi_01}) that
\begin{equation*}
a_i x_i(\alpha_k) + b_i = a_k u_k + b_k - \mu_i(\alpha_k) < a_i u_i b_i = a_i x_k(\alpha_k),
\end{equation*}
which is a contradiction. Thus, it must hold that $x_i(\alpha_k) < u_i$.
\end{proof}

\section{Average execution times of Algorithms~\ref{alg_01} and~\ref{alg_02} and MOSEK}
\label{sec_app_results_tab}

\begin{table}[ht!]
\centering
\resizebox{\textwidth}{!}{
\begin{tabular}{r | cccccccccc}
\toprule
$C$ $\backslash$ $m$ & 1 & 2 & 5 & 10 & 20 & 50 & 100 & 200 & 500 & 1,000 \\
\midrule
1 & $1.07\cdot 10^{-4}$ &$	1.49\cdot 10^{-4}$ &$	2.38\cdot 10^{-4}$ &$	3.68\cdot 10^{-4}$ &$	8.87\cdot 10^{-4}$ &$	1.51\cdot 10^{-3}$ &$	2.94\cdot 10^{-3}$ &$	5.75\cdot 10^{-3}$ &$	1.47\cdot 10^{-2}$ &$	2.75\cdot 10^{-2}$ \\
2& $1.58\cdot 10^{-4}$ &$	2.06\cdot 10^{-4}$ &$	3.93\cdot 10^{-4}$ &$	7.23\cdot 10^{-4}$ &$	1.34\cdot 10^{-3}$ &$	3.29\cdot 10^{-3}$ &$	6.49\cdot 10^{-3}$ &$	1.34\cdot 10^{-2}$ &$	3.13\cdot 10^{-2}$ &$	6.29\cdot 10^{-2}$ \\
5 & $1.75\cdot 10^{-4}$ &$	2.77\cdot 10^{-4}$ &$	5.97\cdot 10^{-4}$ &$	1.08\cdot 10^{-3}$ &$	1.92\cdot 10^{-3}$ &$	4.96\cdot 10^{-3}$ &$	9.91\cdot 10^{-3}$ &$	2.09\cdot 10^{-2}$ &$	4.78\cdot 10^{-2}$ &$	9.89\cdot 10^{-2}$ \\
10 & $2.40\cdot 10^{-4}$ &$	3.72\cdot 10^{-4}$ &$	8.04\cdot 10^{-4}$ &$	1.61\cdot 10^{-3}$ &$	2.89\cdot 10^{-3}$ &$	8.73\cdot 10^{-3}$ &$	1.60\cdot 10^{-2}$ &$	3.22\cdot 10^{-2}$ &$	7.95\cdot 10^{-2}$ &$	1.53\cdot 10^{-1}$ \\
20 & $3.63\cdot 10^{-4}$ &$	5.61\cdot 10^{-4}$ &$	1.41\cdot 10^{-3}$ &$	2.88\cdot 10^{-3}$ &$	5.29\cdot 10^{-3}$ &$	1.37\cdot 10^{-2}$ &$	2.94\cdot 10^{-2}$ &$	5.40\cdot 10^{-2}$ &$	1.33\cdot 10^{-1}$ &$	2.72\cdot 10^{-1}$ \\
50 & $6.63\cdot 10^{-4}$ &$	1.54\cdot 10^{-3}$ &$	3.09\cdot 10^{-3}$ &$	6.56\cdot 10^{-3}$ &$	1.36\cdot 10^{-2}$ &$	3.09\cdot 10^{-2}$ &$	6.35\cdot 10^{-2}$ &$	1.29\cdot 10^{-1}$ &$	3.10\cdot 10^{-1}$ &$	6.10\cdot 10^{-1}$ \\
100 & $1.16\cdot 10^{-3}$ &$	2.84\cdot 10^{-3}$ &$	6.28\cdot 10^{-3}$ &$	1.22\cdot 10^{-2}$ &$	2.33\cdot 10^{-2}$ &$	6.12\cdot 10^{-2}$ &$	1.14\cdot 10^{-1}$ &$	2.20\cdot 10^{-1}$ &$	6.29\cdot 10^{-1}$ &$	1.18\cdot 10^{+0}$ \\
200 & $3.45\cdot 10^{-3}$ &$	4.64\cdot 10^{-3}$ &$	1.38\cdot 10^{-2}$ &$	2.37\cdot 10^{-2}$ &$	4.66\cdot 10^{-2}$ &$	1.10\cdot 10^{-1}$ &$	2.15\cdot 10^{-1}$ &$	4.56\cdot 10^{-1}$ &$	1.11\cdot 10^{+0}$ &$	2.45\cdot 10^{+0}$ \\
500 & $6.08\cdot 10^{-3}$ &$	1.13\cdot 10^{-2}$ &$	2.87\cdot 10^{-2}$ &$	5.78\cdot 10^{-2}$ &$	1.11\cdot 10^{-1}$ &$	2.83\cdot 10^{-1}$ &$	5.29\cdot 10^{-1}$ &$	1.04\cdot 10^{+0}$ &$	3.01\cdot 10^{+0}$ &$	5.43\cdot 10^{+0}$ \\
1,000 & $1.36\cdot 10^{-2}$ &$	2.17\cdot 10^{-2}$ &$	5.84\cdot 10^{-2}$ &$	1.09\cdot 10^{-1}$ &$	2.34\cdot 10^{-1}$ &$	5.43\cdot 10^{-1}$ &$	1.12\cdot 10^{+0}$ &$	2.28\cdot 10^{+0}$ &$	6.89\cdot 10^{+0}$ &$	1.27\cdot 10^{+1}$ \\
\bottomrule
\end{tabular}
}
\caption{Average execution times of Algorithm~\ref{alg_01} for each combination of $C$ and $m$.}
\label{tab_avg_alg01}
\end{table}

\begin{table}[ht!]
\centering
\resizebox{\textwidth}{!}{
\begin{tabular}{r | cccccccccc}
\toprule
$C$ $\backslash$ $m$ & 1 & 2 & 5 & 10 & 20 & 50 & 100 & 200 & 500 & 1,000 \\
\midrule
1 & $7.31\cdot 10^{-5}$ &$	1.11\cdot 10^{-4}$ &$	2.11\cdot 10^{-4}$ &$	3.77\cdot 10^{-4}$ &$	8.84\cdot 10^{-4}$ &$	1.69\cdot 10^{-3}$ &$	3.76\cdot 10^{-3}$ &$	6.90\cdot 10^{-3}$ &$	1.73\cdot 10^{-2}$ &$	3.50\cdot 10^{-2}$ \\
2 & $1.20\cdot 10^{-4}$ &$	1.80\cdot 10^{-4}$ &$	3.79\cdot 10^{-4}$ &$	7.23\cdot 10^{-4}$ &$	1.49\cdot 10^{-3}$ &$	3.69\cdot 10^{-3}$ &$	6.74\cdot 10^{-3}$ &$	1.47\cdot 10^{-2}$ &$	3.48\cdot 10^{-2}$ &$	6.93\cdot 10^{-2}$ \\
5 & $1.47\cdot 10^{-4}$ &$	2.63\cdot 10^{-4}$ &$	5.59\cdot 10^{-4}$ &$	1.08\cdot 10^{-3}$ &$	1.98\cdot 10^{-3}$ &$	4.95\cdot 10^{-3}$ &$	1.01\cdot 10^{-2}$ &$	2.01\cdot 10^{-2}$ &$	5.07\cdot 10^{-2}$ &$	1.03\cdot 10^{-1}$ \\
10 & $1.98\cdot 10^{-4}$ &$	3.55\cdot 10^{-4}$ &$	8.26\cdot 10^{-4}$ &$	1.61\cdot 10^{-3}$ &$	2.99\cdot 10^{-3}$ &$	7.82\cdot 10^{-3}$ &$	1.56\cdot 10^{-2}$ &$	3.01\cdot 10^{-2}$ &$	7.63\cdot 10^{-2}$ &$	1.55\cdot 10^{-1}$ \\
20 & $3.10\cdot 10^{-4}$ &$	5.57\cdot 10^{-4}$ &$	1.35\cdot 10^{-3}$ &$	2.56\cdot 10^{-3}$ &$	5.94\cdot 10^{-3}$ &$	1.27\cdot 10^{-2}$ &$	2.52\cdot 10^{-2}$ &$	5.00\cdot 10^{-2}$ &$	1.26\cdot 10^{-1}$ &$	2.55\cdot 10^{-1}$ \\
50 & $6.44\cdot 10^{-4}$ &$	1.34\cdot 10^{-3}$ &$	3.20\cdot 10^{-3}$ &$	6.23\cdot 10^{-3}$ &$	1.20\cdot 10^{-2}$ &$	2.85\cdot 10^{-2}$ &$	5.74\cdot 10^{-2}$ &$	1.23\cdot 10^{-1}$ &$	2.89\cdot 10^{-1}$ &$	5.84\cdot 10^{-1}$ \\
100 & $1.17\cdot 10^{-3}$ &$	2.17\cdot 10^{-3}$ &$	6.03\cdot 10^{-3}$ &$	1.05\cdot 10^{-2}$ &$	2.14\cdot 10^{-2}$ &$	5.19\cdot 10^{-2}$ &$	1.02\cdot 10^{-1}$ &$	2.06\cdot 10^{-1}$ &$	5.62\cdot 10^{-1}$ &$	1.12\cdot 10^{+0}$ \\
200 & $2.24\cdot 10^{-3}$ &$	4.23\cdot 10^{-3}$ &$	1.29\cdot 10^{-2}$ &$	2.04\cdot 10^{-2}$ &$	4.03\cdot 10^{-2}$ &$	1.00\cdot 10^{-1}$ &$	2.07\cdot 10^{-1}$ &$	4.05\cdot 10^{-1}$ &$	1.05\cdot 10^{+0}$ &$	2.22\cdot 10^{+0}$ \\
500 & $5.69\cdot 10^{-3}$ &$	1.05\cdot 10^{-2}$ &$	2.53\cdot 10^{-2}$ &$	5.26\cdot 10^{-2}$ &$	9.80\cdot 10^{-2}$ &$	2.55\cdot 10^{-1}$ &$	5.06\cdot 10^{-1}$ &$	1.02\cdot 10^{+0}$ &$	2.63\cdot 10^{+0}$ &$	5.42\cdot 10^{+0}$ \\
1,000 & $1.17\cdot 10^{-2}$ &$	2.00\cdot 10^{-2}$ &$	4.98\cdot 10^{-2}$ &$	1.00\cdot 10^{-1}$ &$	2.00\cdot 10^{-1}$ &$	5.14\cdot 10^{-1}$ &$	1.04\cdot 10^{+0}$ &$	2.07\cdot 10^{+0}$ &$	5.39\cdot 10^{+0}$ &$	1.10\cdot 10^{+1}$ \\
\bottomrule
\end{tabular}
}
\caption{Average execution times of Algorithm~\ref{alg_02} for each combination of $C$ and $m$.}
\label{tab_avg_alg02}
\end{table}

\begin{table}[ht!]
\centering
\resizebox{\textwidth}{!}{
\begin{tabular}{r | cccccccccc}
\toprule
$C$ $\backslash$ $m$ & 1 & 2 & 5 & 10 & 20 & 50 & 100 & 200 & 500 & 1,000 \\
\midrule
1 & $1.09\cdot 10^{-2}$ &$	1.26\cdot 10^{-2}$ &$	1.26\cdot 10^{-2}$ &$	1.26\cdot 10^{-2}$ &$	1.45\cdot 10^{-2}$ &$	1.26\cdot 10^{-2}$ &$	4.19\cdot 10^{-2}$ &$	1.89\cdot 10^{-2}$ &$	2.33\cdot 10^{-2}$ &$	2.95\cdot 10^{-2}$ \\
2 & $1.21\cdot 10^{-2}$ &$	1.20\cdot 10^{-2}$ &$	1.26\cdot 10^{-2}$ &$	1.70\cdot 10^{-2}$ &$	1.80\cdot 10^{-2}$ &$	2.29\cdot 10^{-2}$ &$	2.13\cdot 10^{-2}$ &$	2.87\cdot 10^{-2}$ &$	4.29\cdot 10^{-2}$ &$	7.25\cdot 10^{-2}$ \\
5 & $1.21\cdot 10^{-2}$ &$	1.20\cdot 10^{-2}$ &$	1.25\cdot 10^{-2}$ &$	1.24\cdot 10^{-2}$ &$	1.53\cdot 10^{-2}$ &$	2.00\cdot 10^{-2}$ &$	2.83\cdot 10^{-2}$ &$	4.23\cdot 10^{-2}$ &$	7.94\cdot 10^{-2}$ &$	1.36\cdot 10^{-1}$ \\
10 & $1.23\cdot 10^{-2}$ &$	1.27\cdot 10^{-2}$ &$	1.42\cdot 10^{-2}$ &$	1.57\cdot 10^{-2}$ &$	2.61\cdot 10^{-2}$ &$	3.55\cdot 10^{-2}$ &$	4.60\cdot 10^{-2}$ &$	7.39\cdot 10^{-2}$ &$	1.57\cdot 10^{-1}$ &$	3.15\cdot 10^{-1}$ \\
20 & $1.27\cdot 10^{-2}$ &$	1.38\cdot 10^{-2}$ &$	1.68\cdot 10^{-2}$ &$	2.09\cdot 10^{-2}$ &$	3.31\cdot 10^{-2}$ &$	5.26\cdot 10^{-2}$ &$	8.72\cdot 10^{-2}$ &$	1.54\cdot 10^{-1}$ &$	3.65\cdot 10^{-1}$ &$	8.12\cdot 10^{-1}$ \\
50 & $1.52\cdot 10^{-2}$ &$	1.98\cdot 10^{-2}$ &$	4.20\cdot 10^{-2}$ &$	5.41\cdot 10^{-2}$ &$	7.62\cdot 10^{-2}$ &$	1.61\cdot 10^{-1}$ &$	3.27\cdot 10^{-1}$ &$	6.35\cdot 10^{-1}$ &$	1.55\cdot 10^{+0}$ &$	3.27\cdot 10^{+0}$ \\
100 & $2.34\cdot 10^{-2}$ &$	4.32\cdot 10^{-2}$ &$	7.07\cdot 10^{-2}$ &$	1.20\cdot 10^{-1}$ &$	2.54\cdot 10^{-1}$ &$	5.83\cdot 10^{-1}$ &$	1.21\cdot 10^{+0}$ &$	2.15\cdot 10^{+0}$ &$	5.85\cdot 10^{+0}$ &$	1.20\cdot 10^{+1}$ \\
200 & $5.70\cdot 10^{-2}$ &$	1.00\cdot 10^{-1}$ &$	2.21\cdot 10^{-1}$ &$	4.42\cdot 10^{-1}$ &$	8.32\cdot 10^{-1}$ &$	2.03\cdot 10^{+0}$ &$	4.21\cdot 10^{+0}$ &$	8.73\cdot 10^{+0}$ &$	2.40\cdot 10^{+1}$ &$	4.41\cdot 10^{+1}$ \\
500 & $3.08\cdot 10^{-1}$ &$	5.92\cdot 10^{-1}$ &$	1.55\cdot 10^{+0}$ &$	3.30\cdot 10^{+0}$ &$	5.98\cdot 10^{+0}$ &$	1.63\cdot 10^{+1}$ &$	2.98\cdot 10^{+1}$ &$	6.20\cdot 10^{+1}$ &$	1.60\cdot 10^{+2}$ &$	3.51\cdot 10^{+2}$ \\
1,000 & $1.06\cdot 10^{+0}$ &$	3.05\cdot 10^{+0}$ &$	7.37\cdot 10^{+0}$ &$	1.54\cdot 10^{+1}$ &$	3.41\cdot 10^{+1}$ &$	7.84\cdot 10^{+1}$ &$	1.71\cdot 10^{+2}$ &$	3.75\cdot 10^{+2}$  & - & - \\
\bottomrule
\end{tabular}
}
\caption{Average execution times of MOSEK for each combination of $C$ and $m$.}
\label{tab_avg_mosek}
\end{table}

\bibliographystyle{abbrv}

\bibliography{C:/Users/SchootUiterkampMHH/Documents/Papers/library}

\end{document}